\documentclass[a4paper, reqno, 11pt, notitlepage]{amsart}
    \usepackage[margin=1in]{geometry}  
    \usepackage[utf8]{inputenc}
    \usepackage[T1]{fontenc}
    \usepackage[usenames]{color}
    \usepackage{amssymb,mathrsfs,extarrows,nicefrac,mathtools,framed,tikz-cd,multicol,
    relsize,dsfont,dutchcal,euscript} 
    \usepackage[shortlabels]{enumitem}
    \usepackage[bbgreekl]{mathbbol}
    \usetikzlibrary{shapes.geometric, positioning, decorations.pathreplacing}
    \usepackage[activate={true,nocompatibility},final,tracking=true,kerning=true,spacing=true]{microtype}
    \microtypecontext{spacing=nonfrench}

    \usepackage[pagebackref]{hyperref} 
        \definecolor{darkblue}{rgb}{0,0,.85} 
        \definecolor{darkred}{rgb}{0.84,0,0}
        \hypersetup{colorlinks = true,
                linkcolor = darkblue,
                urlcolor  = darkblue,
                citecolor = darkred,
                anchorcolor = darkblue}

\makeatletter

\renewcommand{\tocsection}[3]{%
  \indentlabel{\@ifnotempty{#2}{\bfseries\ignorespaces#1 #2\quad}}\bfseries#3}

\renewcommand{\tocsubsection}[3]{%
  \indentlabel{\@ifnotempty{#2}{\ignorespaces#1 #2\quad}}#3} 

\newcommand\@dotsep{4.5}
\def\@tocline#1#2#3#4#5#6#7{\relax
  \ifnum #1>\c@tocdepth 
  \else
    \par \addpenalty\@secpenalty\addvspace{#2}%
    \begingroup \hyphenpenalty\@M
    \@ifempty{#4}{%
      \@tempdima\csname r@tocindent\number#1\endcsname\relax
    }{%
      \@tempdima#4\relax
    }%
    \parindent\z@ \leftskip#3\relax \advance\leftskip\@tempdima\relax
    \rightskip\@pnumwidth plus1em \parfillskip-\@pnumwidth
    #5\leavevmode\hskip-\@tempdima{#6}\nobreak
    \leaders\hbox{$\m@th\mkern \@dotsep mu\hbox{.}\mkern \@dotsep mu$}\hfill
    \nobreak
    \hbox to\@pnumwidth{\@tocpagenum{\ifnum#1=1\bfseries\fi#7}}\par
    \nobreak
    \endgroup
  \fi}
\AtBeginDocument{%
\expandafter\renewcommand\csname r@tocindent0\endcsname{0pt}
}
\def\l@subsection{\@tocline{2}{0pt}{2.5pc}{5pc}{}}
\makeatother

\makeatletter
\def\subsubsection{\@startsection{subsubsection}{3}%
  \z@{.5\linespacing\@plus.7\linespacing}{-.5em}%
  {\normalfont\bfseries}}
\makeatother

    \makeatletter
    \def\paragraph{\@startsection{paragraph}{4}%
    \z@\z@{-\fontdimen2\font}%
    {\normalfont\bfseries}}
    \makeatother

  \theoremstyle{plain}
  \newtheorem{thm}{Theorem}[section]
  \newtheorem{lemma}[thm]{Lemma}
  \newtheorem{cor}[thm]{Corollary}
  \newtheorem{prop}[thm]{Proposition}
  \newtheorem*{prop*}{Proposition}
  \theoremstyle{definition}
  \newtheorem{defin}[thm]{Definition}
  \newtheorem{term}[thm]{Terminology}
  \newtheorem{example}[thm]{Example}
  \newtheorem{question}[thm]{Question}
  \newtheorem{remark}[thm]{Remark}
  \newtheorem{warning}[thm]{Warning}
  \newtheorem{thmi}[thm]{Theorem}
  \newtheorem{propi}[thm]{Proposition}
  \newtheorem{nota}[thm]{Notation}

  \numberwithin{equation}{subsection}

  \newcommand{\isomto}{\xlongrightarrow{\,\smash{\raisebox{-0.65ex}{\ensuremath{\displaystyle\sim}}}\,}}

  \newcommand{\mc}[1]{\EuScript{#1}}
  \newcommand{\mf}[1]{\mathfrak{#1}}
  \newcommand{\ms}[1]{\mathscr{#1}}
  \newcommand{\mr}[1]{\mathrm{#1}}
  \newcommand{\bb}[1]{\mathbb{#1}}
  \newcommand{\cat}[1]{\mathbf{#1}}
  \newcommand{\wh}[1]{\widehat{#1}}
  \newcommand{\ov}[1]{\overline{#1}}
  \newcommand{\wt}[1]{\widetilde{#1}}
  \newcommand{\et}{\mathrm{\acute{e}t}}
  \newcommand{\Et}{\mathrm{\acute{E}t}}
  \newcommand{\rig}{\mathrm{rig}}
  
  \newcommand{\an}{\mathrm{an}}
  \newcommand{\FEt}{\cat{F\acute{E}t}}
  \DeclareMathOperator{\Spf}{Spf}
  \DeclareMathOperator{\Spa}{Spa}
  \DeclareMathOperator{\Spec}{Spec}
  \DeclareMathOperator{\Hom}{Hom}
  \newcommand{\mct}{\mathcal{t}}
  \newcommand{\mcg}{\mathcal{g}}
  \newcommand{\mcO}{\mathcal{O}}
 \newcommand{\mcF}{\mathcal{F}}
    \newcommand{\mcQ}{\mathcal{Q}}

  \newcommand{\ip}{[\nicefrac{1}{\pi}]} 

  \newcommand{\cO}{\mathcal{O}}

\newcommand{\mysetminusD}{\hbox{\tikz{\draw[line width=0.6pt,line cap=round] (3pt,0) -- (0,6pt);}}}
\newcommand{\mysetminusT}{\mysetminusD}
\newcommand{\mysetminusS}{\hbox{\tikz{\draw[line width=0.45pt,line cap=round] (2pt,0) -- (0,4pt);}}}
\newcommand{\mysetminusSS}{\hbox{\tikz{\draw[line width=0.4pt,line cap=round] (1.5pt,0) -- (0,3pt);}}}

\newcommand{\mysetminus}{\mathbin{\mathchoice{\mysetminusD}{\mysetminusT}{\mysetminusS}{\mysetminusSS}}}

  \newcommand{\stacks}[2][Tag]{\cite[\href{https://stacks.math.columbia.edu/tag/#2}{#1 #2}]{StacksProject}}

  \author{Piotr Achinger}
  \address{Piotr Achinger, Instytut Matematyczny PAN, Śniadeckich 8, 00-656 Warsaw, Poland}
  \email{pachinger@impan.pl}

  \author{Alex Youcis}
  \address{Alex Youcis, Department of Mathematics, National University of Singapore, Level 4, Block S17, 10
Lower Kent Ridge Road, Singapore, 119076}
  \email{alex.youcis@gmail.com}

  \title{Beauville--Laszlo Gluing of Algebraic Spaces}
  \date{\today}

\begin{document}

\begin{abstract} 
    For a complete discrete valuation field $K$, we show that one may always glue a~separated formal algebraic space $\mf{X}$ over $\mcO_K$ to a separated algebraic space $U$ over $K$ along an open immersion of rigid spaces $j\colon \mf{X}^{\rm rig}\to U^\mr{an}$, producing a separated algebraic space $X$ over $\mcO_K$. This process gives rise to an equivalence between such `gluing triples' $(U,\mf{X},j)$ and separated algebraic spaces $X$ over $\mcO_K$, which one might interpret as a version of the Beauville--Laszlo theorem for algebraic spaces rather than coherent sheaves. Moreover, an analogous equivalence exists over any excellent base. Examples due to Matsumoto imply that the result of such a~gluing might be a genuine algebraic space (not a scheme) even if $U$ and the special fiber of $\mf{X}$ are projective. The proof is a combination of Nagata compactification theorem for algebraic spaces and of Artin's contraction theorem. We give multiple examples and applications of this idea. 
\end{abstract}

\maketitle

\tableofcontents

\section{Introduction}
\label{s:intro}

In this article we deal with the following basic problem. Let $K$ be a non-archimedean field with valuation ring $\mcO$, and let $U$ be a separated scheme locally of finite type over $K$. By a \emph{model} of $U$ over $\mcO$ we shall mean a separated scheme of locally finite type $X$ over $\mcO$ with generic fiber $X_K \simeq U$. The question is: 
\begin{equation*}
    \text{\emph{How can we economically describe all models $X$ of $U$?}}
\end{equation*}
Our answer uses formal and rigid geometry, and is hinted at in multiple earlier works (e.g., \cite{ArtinII}, \cite{BBT}, and \cite{BLNeron}). To illustrate the idea, let us consider a basic example from Bruhat--Tits theory.

\begin{example}[Models of ${\rm GL}_n$]\label{eg:BT}
Let $G = \mathrm{GL}(V)$ for a finite dimensional vector space $V$ over $K$. We aim to describe all smooth affine group schemes $\mathcal{G}$ over $\cO$ with generic fiber $G$. There are many interesting examples of such $\mathcal{G}$.
\begin{itemize}
    \item An $\mcO$-lattice $\Lambda\subseteq V$ gives rise to the model $\mathcal{G} = \mathrm{Aut}(\Lambda)$ of $G$ over $\mcO$.
    \item If $K$ is discretely valued with uniformizer~$\pi$, we have the \emph{Iwahori model} $\mathcal{I}$ with
    \[ 
        \mathcal{I}(\cO) =  \{(a_{ij})\in {\rm GL}_n(\cO)\,:\, a_{ij}\in (\pi) \text{ when }i<j\}.
    \]
\end{itemize}
A choice of $\mathcal{G}$ gives rise to rigid-geometric data: the image of $\mathcal{G}(\mcO)$ in $G(K)$ is the set of $K$-points of a rigid-analytic affinoid (i.e., ``compact'') subgroup $\mathsf{G}$ of  $G^{\rm an}$. In the first example above, for $V = K^n$ and  $\Lambda\simeq \mcO^n$, we have
\[ 
    \mathsf{G} = \{(a_{ij})\in {\rm GL}_n\,:\, |a_{ij}|\leqslant 1\},
\]
while for the Iwahori model we obtain the \emph{Iwahori subgroup} of $G(K)$
\begin{equation*}
    \mathsf{I}(K) = \{(a_{ij})\in {\rm GL}_n(K)\,:\, |a_{ij}|\leqslant 1\text{ and }|a_{ij}|\leqslant|
    \pi|\text{ when }i<j\}. 
\end{equation*}
Furthermore, $\mathcal{G}$ allows us to define a reduction map $\mathsf{G}(K)=\mathcal{G}(\cO)\to \mathcal{G}(k)$, where $k$ is the residue field of $K$. One may think of $\mathcal{G}$ as the result of `gluing' $G$ to $\mathcal{G}_k$ along $\mathsf{G}$,  with the `glue' being provided by $\mathsf{G}$ and the reduction map.  In fact, all smooth affine models of $G$ arise from such gluings (cf.\@ \cite[Corollary 2.10.11]{KalethaPrasad}, which implies that $\mathsf{G}$ determines $\mathcal{G}$).
\end{example}

The procedure of Example \ref{eg:BT} applies in general when appropriately formulated. To a scheme $X$ locally of finite type over $\mcO$, one attaches a triple $\mct(X)=(X_K, \widehat{X}, j_X)$ consisting of
\begin{itemize}
    \item its generic fiber $X_K$,
    \item its $\pi$-adic formal completion $\widehat{X}$ (where $\pi$ is a pseudouniformizer of $K$),
    \item the natural morphism  of rigid analytic spaces over $K$ (see Proposition \ref{prop:j-map})
    \[ 
        j_X \colon \widehat{X}^{\rm rig} \to X_K^{\rm an},
    \]
    where $\widehat{X}^{\rm rig}$ is the rigid generic fiber of $\widehat{X}$ and $X_K^{\rm an}$ is the analytification of $X_K$.
\end{itemize}
We treat $\mct(X)$ as an object of the category of triples $(U,\mf{X},j)$ consisting of a~$K$-scheme $U$, a formal scheme $\mf{X}$ over $\mcO$, and a morphism of rigid spaces $j\colon \mf{X}^\mr{rig}\to U^\mr{an}$.
Intuitively, $X$ should be described as the effect of gluing  $\wh{X}$ to $X_K$ along  $\wh{X}^\mr{rig}$. To connect this to Example \ref{eg:BT}, observe that $\mathsf{G}=\wh{\mathcal{G}}^\mr{rig}$ and the reduction map $\mathsf{G}(K)\to \mathcal{G}(k)$ is obtained by passing to $K$-points from the specialization map $\mr{sp}\colon \wh{\mathcal{G}}^\mr{rig}\to \mathcal{G}_k$ of the formal scheme $\wh{\mathcal{G}}$.

While the pushout $X_K\sqcup_{\wh{X}^\mr{rig}}\wh{X}$ does not literally make sense, we can show the following result also previously observed (in lesser generality) in \cite{BLNeron} and \cite{IKY1}.

\begin{propi}[{see Corollary \ref{cor:fully-faithfulness-schemes}}] 
    The functor $\mct$ is fully faithful.
\end{propi}

It is a natural question if, or to what extent, the functor $t$ is an equivalence. Let us only consider separated schemes, and let the target category be that of triples $(U, \mf{X}, j)$ as above but with both $U$ and $\mf{X}$ separated, and where the map $j$ is an open embedding. In \cite[Example 5.2]{Matsumoto} one finds a smooth proper \emph{algebraic space} $X$ over $\mathbb{Z}_p$ such that both $X_{\mathbb{Q}_p}$ and $X_{\mathbb{F}_p}$ are projective schemes (K3 surfaces), and which is not a scheme. One may deduce that the corresponding triple $(X_{\mathbb{Q}_p}, \wh{X}, j_X)$ is not in the essential image of $\mct$. This example shows that the question is more naturally formulated in the realm of algebraic spaces. Somewhat surprisingly, it is always possible to `glue' such triples into algebraic spaces. 

To state our result, we require some setup. Let $S$ be an excellent algebraic space and $S_0\subseteq S$ a closed subspace (e.g., $(S,S_0)=(\Spec(\mcO),V(\pi))$ when $K$ is discretely valued). Let $\widehat{S}$ be the formal completion of $S$ along $S_0$, and let $S^\circ = S\mysetminus S_0$. For an algebraic space $X$ over $S$ we write $X^\circ = X\times_S S^\circ$. Denote by $\cat{AlgSp}^{\rm sep}_S$ the category of algebraic spaces $X$ separated and locally of finite type over $S$. Let $\cat{Trip}^\mr{sep}_{(S, S_0)}$ denote the category of \emph{separated gluing triples} $(U, \mf{X}, j)$ where:
\begin{itemize} 
    \item $U$ is a algebraic space separated and locally of finite type over  $S^\circ$, 
    \item $\mf{X}$ is a formal algebraic space separated and locally of finite type over $\wh{S}$,
    \item $j\colon \mf{X}^\mr{rig}\to U^\mr{an}$ is an open embedding of rigid algebraic spaces over $\wh{S}^{\rm rig}$,\footnote{In fact, by a result of Conrad--Temkin in \cite{ConradTemkin} and Warner in \cite{Warner} (also announced in \cite{FujiwaraKato}), under the separatedness hypotheses both $\mf{X}^\mr{rig}$ and $U^\mr{an}$ are representable by rigid spaces. Therefore, if the reader so wishes, subtleties of adic algebraic spaces can be largely ignored in this paper.}
\end{itemize}
(see \S\ref{ss:formal-schemes-rigid-locus} for a recollection of these concepts in this generality). One might picture a gluing triple as in Figure~\ref{fig:fried-egg} below.

\begin{figure}[ht]
    \centering
    \includegraphics{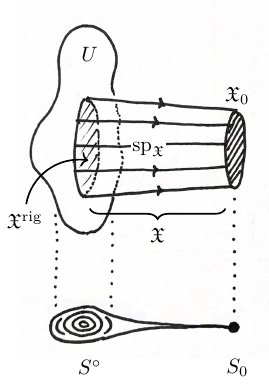}
    \caption{A picture of the gluing triple $(U,\mf{X},j)$.}
    \label{fig:fried-egg}
\end{figure}

The precise version of the claim that one can always uniquely glue a separated gluing triple $(U,\mf{X},j)$ together into some algebraic space $X$ is the following.

\begin{thmi}[{Beauville--Laszlo Gluing of algebraic spaces, see Theorem \ref{thm:main}}]\label{thmi:main}
    The functor
    \[ 
        \mct \colon \cat{AlgSp}^{\rm sep}_S \longrightarrow \cat{Trip}^\mr{sep}_{(S, S_0)},
        \qquad
        \mct(X) = (X^\circ, \widehat{X}, j_X)
    \]
    is an equivalence of categories. 
\end{thmi}

Theorem~\ref{thmi:main} is quite clarifying with respect to several well-known phenomena in arithmetic geometry. For example:
\begin{itemize}[$\diamond$]
    \item questions of good reduction over $K$ vs.\@ $\wh{K}$ (see Example \ref{eg:Poonen} and Proposition \ref{prop:models-over-K-vs-hat-K}), 
    \item the construction and behavior of N\'eron models (see Example \ref{eg:Neron-model-Gm} and Remark~\ref{rem:BLNeron}), \item algebraizability of formal schemes (see Example~\ref{eg:DR} and Proposition \ref{prop:A-E-equiv}),
    \item questions related to the existence of a specialization morphism for the \'etale fundamental group beyond the proper case (see Proposition \ref{prop:FEt-algebraizability} and Remark \ref{rem:Grothendieck-specialization}).
\end{itemize}
In addition, the ideas surrounding it play a central role in \cite{IKY2}, and appear implicitly in Bruhat--Tits theory (see \cite[\S 2.10]{KalethaPrasad}). 

\subsubsection*{Relationship to the classical Beauville--Laszlo theorem}
To explain the terminology `Beauville--Laszlo gluing', let us recall the statement of the Beauville--Laszlo theorem from \cite{BeauvilleLaszlo}. Let $A$ be a ring with a non-zerodivisor $\pi$. Then, the theorem asserts that the functor from the category of $\pi$-torsionfree $A$-modules $M$, to the category of triples $(F,G,\iota)$ where 
\begin{enumerate}[(i)]
    \item $F$ is a $A[\nicefrac{1}{\pi}]$-module,
    \item $G$ is a $\pi$-torsionfree $\wh{A}$-module,
    \item $\iota\colon F\otimes_{A[\nicefrac{1}{\pi}]}\wh{A}[\nicefrac{1}{\pi}]\isomto G[\nicefrac{1}{\pi}]$ is an isomorphism of $\wh{A}[\nicefrac{1}{\pi}]$-modules,
\end{enumerate}
given by sending $M$ to $(M[\nicefrac{1}{\pi}],\wh{M},\iota_M)$ (where $\iota_M$ is the natural isomorphism) is an equivalence. One can extend this result to more general pairs $(A,\pi)$, and to the setting where $M$ is replaced by an affine scheme $\Spec(R)$ over $\Spec(A)$ (e.g., see \stacks[Lemma]{0F9Q}).\footnote{We use the name `Beauville--Laszlo gluing' here because it has become the established name for such results. That said, similar results were previously obtained by Artin in \cite{ArtinII} with Noetherian hypotheses, even allowing the role of $\pi$ to be played by an arbitrary ideal $I$. }

There is simplicity in the Beauville--Laszlo method as it avoids needing to explicitly think about formal geometry, using $\Spec(\wh{R})$ instead of $\Spf(\wh{R})$. That said, this prevents one from being able to globalize such gluing procedures on a general scheme $X=\bigcup \Spec(R_i)$ over $\Spec(A)$, as the correct gluing cannot happen between $\Spec(\wh{R}_i)$ and $\Spec(\wh{R}_j)$, but only between $\Spf(\wh{R}_i)$ and $\Spf(\wh{R}_j)$. Similar problems arise when trying to globalize the base $\Spec(A)$. This forces the introduction of formal geometry and in turn the introduction of rigid geometry, as the analogue of the isomorphism in (iii) above must now take place over the locus $\{\pi\ne 0\}$ in $\Spf(\wh{A})$ which only exists in the world of rigid analytic geometry. 

Despite this connection, we do emphasize that our gluing results differ from the classical version of Beauville--Laszlo gluing outside of the coherent situation, i.e., when $\Spec(R)$ is not a finite scheme over $\Spec(A)$. Indeed, while $\Spec(\wh{R}[\nicefrac{1}{\pi}])$ has a close connection to $\Spa(\wh{R}[\nicefrac{1}{\pi}])$ in the finite case (e.g., the latter is the analytification of the former), the difference becomes drastic even for finite type (but not finite) affine schemes over $\Spec(A)$. For instance, $\Spec(\mcO[x^{\pm 1}][\nicefrac{1}{\pi}])^\mr{an}=\mathbb{G}_{m,K}^\mr{an}$ but the $\{\pi\ne 0\}$ locus in $\Spf(\mcO[x^{\pm 1}]^\wedge)$ is the circle group $\{x\in \mathbb{G}_{m,K}^\mr{an}:|x|=1\}$. 

\subsubsection*{Idea of proof and conditions on $S$} Many of our results, including the fully faithfulness portion of Theorem \ref{thmi:main} requires only that $(S,S_0)$ is of Type (N)/(V) (i.e., $S$ is Noetherian or the spectrum of a complete rank one valuation ring). But, our proof of essential surjectivity uses
\begin{enumerate}[(a)]
    \item the Artin contraction theorem (see Theorem \ref{thm:Artin-contraction}) which (with some work) handles the case when the gluing triple is `proper-like', and
    \item Nagata compactification for algebraic spaces (see \cite{CLO}) to reduce to the `proper-like' case.
\end{enumerate}
Roughly, (a) explains our restriction to $S$ excellent, and (b) explains our separation hypotheses. It would be worthwhile to try and relax either of these hypotheses (e.g., allowing $S = \Spec(\mcO_{\mathbb{C}_p})$).

\subsubsection*{Outline of the article} 

In \S\ref{ss:formal-schemes-rigid-locus}--\ref{ss:gluing-triple-category} we review the necessary background in rigid geometry to define the category of gluing triples in the generality we require. In \S\ref{s:examples} we discuss some illustrative examples of triples and give some applications of Theorem \ref{thmi:main} to clarify them. In particular, in Proposition~\ref{prop:affine-triples}, we explain how to single out those gluing triples whose gluing is a scheme (opposed to an algebraic space), generalizing results from \cite{BLNeron}. In \S\ref{s:coh-gluing} we prove gluing for coherent sheaves: that for an algebraic space $X$, coherent sheaves on $X$ and $\mct(X)$ are the same. In \S\ref{s:proof-of-gluing-theorem} we supply proofs of Theorem~\ref{thmi:main} and two other lengthier propositions.

\subsubsection*{\texorpdfstring{$\heartsuit$ Acknowledgments}{Acknowledgments}}

The authors would like to thank Ofer Gabber,  David Hansen, Aise Johan de Jong, and Martin Olsson, for helpful discussions.

The first author (PA) was supported by the project KAPIBARA funded by the European Research Council (ERC) under the European Union's Horizon 2020 research and innovation programme (grant agreement No 802787). Part of this work was conducted while the second author (AY) was visiting the Hausdorff Research Institute for Mathematics, funded by the Deutsche Forschungsgemeinschaft (DFG, German Research Foundation) under Germany's Excellence Strategy – EXC-2047/1 – 390685813. Part of this work was also carried out while the second author (AY) was supported by JSPS KAKENHI Grant Number 22F22323. 

\subsubsection*{Notation and conventions}

\begin{itemize}
    \item All \emph{(formal) algebraic spaces} in this article are assumed \textbf{quasi-separated}. 
    \item A \emph{non-archimedean field} is a field which is complete with respect to a rank one valuation. We denote by $\cO_K$ or simply by $\cO$ the valuation ring of a non-archimedean field $K$.
    \item For a Huber ring $A$, we shorten the notation $\Spa(A,A^\circ)$ to $\Spa(A)$.
    \item A locally spectral space is called \emph{coherent} if it is quasi-compact and quasi-separated,
    \item For categories of schemes, formal schemes, algebraic spaces etc.\ over a fixed base, we shall use the following abbreviations (as superscripts) to denote properties specifying the corresponding full subcategories:
    \begin{multicols}{2}
    \begin{itemize}[$\diamond$]
        \item lft: locally of finite type, 
        \item adm: admissible (i.e., locally of finite type and $\mc{I}$-torsion free), 
        \item ft: finite type, 
        \item sep: separated
        \item coh: coherent.
    \end{itemize}
    \end{multicols}
    \item We shall almost always use different types of letters/fonts to denote objects over different spaces. The letters $X,Y,Z$ will denote algebraic spaces over some scheme $S$, the letters $U,V,W$ will denote algebraic spaces over an open subscheme $S^\circ\subseteq S$, the letters $\mf{X},\mf{Y},\mf{Z}$ will denote formal algebraic spaces over a completion $\wh{S}$ of $S$, and letters like $\mathsf{X},\mathsf{Y},\mathsf{Z}$ will denote rigid algebraic spaces over the rigid locus $\wh{S}^\mr{rig}$. If $X$ is an algebraic space over $S$ we shall usually denote $X\times_S S^\circ$ by $X^\circ$. 
\end{itemize}

\section{Gluing triples}
\label{s:gluing-triples}

In this section, we formalize the notion of gluing triples over a general base $S$, define the `associated gluing triple' of an algebraic space over $S$, and establish some basic properties of such objects. We then formulate the main result of this article: Beauville--Laszlo Gluing of algebraic spaces (see Theorem \ref{thm:main}).

\subsection{Formal and rigid algebraic spaces}
\label{ss:formal-schemes-rigid-locus}

In this subsection, we recall the definitions of the geometric objects appearing in the definition of a gluing triples.

\subsubsection*{Formal schemes and formal algebraic spaces} 

Our references for formal algebraic spaces are \cite{FujiwaraKato} and \stacks[Chapter]{0AHW}. We refer the reader to these references for any undefined terms. 
But unlike in \cite{FujiwaraKato}, in this article all formal schemes and formal algebraic spaces are quasi-separated and (Zariski or \'etale) locally of the form ${\rm Spf}(A)$ for a ring $A$ which is $I$-adically complete and separated for a finitely generated ideal $I\subseteq A$, endowed with the $I$-adic topology. 

We further always assume our formal algebraic spaces $\mf{S}$ are \emph{locally universally rigid-Noetherian} as in \cite[Chapter I, Definitions 2.1.7 and 6.5.1]{FujiwaraKato}, i.e.\ \'etale locally of the form $\Spf(A)$ as above such that the schemes $\Spec(A\langle x_1, \ldots, x_n\rangle)\mysetminus V(I)$ are Noetherian for all $n\geqslant 0$. We shall often further assume $\mf{S}$ is either 
\begin{itemize}[$\diamond$,leftmargin=.6cm]
    \item of \emph{type (N)}: admits an \'etale cover $\mf{U}\to\mf{S}$ where $\mf{U}$ is a locally Noetherian formal scheme, or
    \item of \emph{type (V)}: admits an \'etale cover $\mf{U}\to\mf{S}$ where $\mf{U}$ is a formal scheme locally of finite type over a rank one valuation ring.
\end{itemize}
Note that every algebraic space of type (N)/(V) is locally universally rigid-Noetherian (see \cite[Chapter I, Example 2.1.4]{FujiwaraKato}).

\begin{nota} 
    Let $\mc{I}$ be an ideal sheaf of definition (see \cite[Chapter I, Definitions 1.1.18 and 6.3.16]{FujiwaraKato}) of a coherent formal algebraic space $\mf{S}$. We write $\mf{S}_n$ for $V(\mc{I}^{n+1})\subseteq \mf{S}$, leaving $\mc{I}$ implicit. 
\end{nota}

We next set our notation for categories of formal schemes over a fixed base $\mf{S}$.

\begin{nota} 
    Denote the category of formal schemes (resp.\ formal algebraic spaces) over $\mf{S}$ by $\cat{FSch}_{\mf{S}}$ (resp.\ $\cat{FAlgSp}_{\mf{S}}$). Let $\cat{FSch}^\ast_\mf{S}$ or $\cat{FAlgSp}^\ast_\mf{S}$ be the subcategories with the same objects but only adic morphisms (see \cite[Chapter I, Definitions 1.3.1 and 6.3.18]{FujiwaraKato}). If $\mf{S}$ is a scheme, let $\cat{Sch}_\mf{S}$ and $\cat{AlgSp}_\mf{S}$ be the category of schemes and algebraic spaces over $\mf{S}$, respectively.
\end{nota}

We next recall our notation for the underlying topological space of a formal algebraic space.  

\begin{defin}
    Let $\mf{S}$ be a formal algebraic space. We define the \emph{underlying space} to have set
    \begin{equation*}
        |\mf{S}|= \left\{x\colon \Spec(k_x)\to \mf{S}: k_x\text{ is a field}\right\}/\sim,
    \end{equation*}
    where $x\sim y$ if they can be dominated by a common $z\colon \Spec(k_z)\to \mf{S}$. As in \stacks[Lemma]{03BX}, we may uniquely, functorially topologize this set so that it agrees with the usual underlying space in the representable case.\footnote{In fact, $|\mf{S}|=|\mf{S}_{\rm red}|$ where $\mf{S}_{\mr{red}}$ is the reduced algebraic subspace $\mf{S}_{\rm red}$ (see \stacks[Section]{0GB5}).}
\end{defin}

Finally, we recall that for a closed algebraic subspace $S_0$ of an algebraic space $S$, one may form the \emph{completion} denoted $\wh{S}$ (leaving the role of $S_0$ implicit) as in \cite[Chapter II, \S6.3.(f)]{FujiwaraKato}). This defines a functor $\wh{(-)}\colon \cat{AlgSp}_S\to \cat{FAlgSp}^\ast_{\wh{S}}$, where we complete $X\to S$ along $X_0=X\times_S S_0$.

\subsubsection*{Adic spaces and adic algebraic spaces} 

In this article we use the theory of adic spaces as our foundation for rigid geometry with our main reference being \cite{HuberEC}. That said, we expand the category of adic spaces under consideration to include those adic spaces $\mathsf{S}$ which are locally strongly rigid-Noetherian, i.e., such that every point has an affinoid open neighborhood $\Spa(A,A^+)\subseteq \mathsf{S}$ where $A$ is strongly rigid-Noetherian in the sense of \cite[Definition 2.8]{ZavyalovSheafiness}. 

To simplify terminology and notation, we refer to such locally strongly rigid-Noetherian adic spaces as just `adic spaces'. Also, for any category of adic spaces, an asterisk as a superscript refers to restricting to the wide subcategory consisting of adic morphisms (in the sense of \cite[\S3]{HuberGen})

\begin{defin}\label{def:analytic-locus} 
    A point $s$ of an adic space $\mathsf{S}$ is \emph{analytic} if the topological field $k(s)$ is non-discrete. The set $\mathsf{S}_\mr{a}\subseteq \mathsf{S}$ of analytic points of $\mathsf{S}$ forms an open subset which we call the \emph{analytic locus}, and we say that $\mathsf{S}$ is \emph{analytic} if $\mathsf{S}=\mathsf{S}_\mr{a}$. The analytic locus forms a functor
    \begin{equation}\label{eq:analytic-locus-functor}
    (-)_\mr{a}\colon \left\{\begin{matrix}\text{Adic}\\ \text{spaces}\end{matrix}\right\}^\ast \to  \left\{\begin{matrix}\text{Analytic adic}\\ \text{spaces}\end{matrix}\right\},
    \end{equation}
    which is right adjoint to the inclusion of analytic adic spaces into the category of adic spaces with adic morphisms. 
\end{defin}

We come to our definition of algebraic spaces in the adic space world.

\begin{defin} 
    An \emph{adic algebraic space} is a sheaf $\mathsf{S}$ on the big \'etale site of adic spaces such that 
    \begin{enumerate}[(i)]
        \item the diagonal map $\Delta_\mathsf{S}\colon \mathsf{S}\to\mathsf{S}\times\mathsf{S}$ is representable by adic spaces and quasi-compact,
        \item there exists an \'etale cover $\mathsf{U}\to \mathsf{S}$ where $\mathsf{U}$ is an adic space.
    \end{enumerate} 
    We say an adic algebraic space $\mathsf{S}$ is \emph{analytic} if it admits an \'etale cover $\mathsf{U}\to \mathsf{X}$ with $\mathsf{U}$ analytic. 
\end{defin}
We often assume that our analytic algebraic space $\mathsf{S}$ is of one of the following types:
\begin{itemize}[$\diamond$,leftmargin=.6cm]
    \item \emph{type (N)}: admits an \'etale cover $\mathsf{U}\to\mathsf{S}$ where $\mathsf{U}$ is an adic space locally of the form $\Spa(A,A^+)$ where $A$ has a Noetherian ring of definition,
    \item \emph{type (V)}: admits an \'etale cover $\mathsf{U}\to\mathsf{S}$ where $\mathsf{U}$ is an adic space locally of finite type over some non-archimedean field.
\end{itemize}

\noindent We will be particularly interested in certain categories of adic algebraic spaces, and so we give specific notation to them. 

\begin{nota} 
   For an analytic algebraic space $\mathsf{S}$, denote by $\cat{RigAlgSp}_{\mathsf{S}}$ the category of adic algebraic spaces locally of finite type over $\mathsf{S}$, whose objects are \emph{rigid algebraic spaces over $\mathsf{S}$}.
\end{nota}

Our definition of the underlying topological space is as in the case of formal algebraic space.

\begin{defin} 
    Let $\mathsf{S}$ be an adic algebraic space. We define the \emph{underlying space} to have set
    \begin{equation*}
        |\mathsf{S}|= \left\{x\colon \Spa(k_x,k_x^+)\to \mathsf{S}:(k_x,k_x^+)\text{ is an affinoid field}\right\}
    \end{equation*}
    where $x\sim y$ if they can be dominated by a common $z\colon \Spa(k_z,k_z^+)\to \mathsf{S}$. As in \stacks[Lemma]{03BX}, we may uniquely, functorially topologize this set so that it agrees with the usual underlying space in the representable case.
\end{defin}

Finally, we extend Definition \ref{def:analytic-locus} to the case of adic algebraic spaces.

\begin{defin}
    If $\mathsf{S}$ is an adic algebraic space, a point of $|\mathsf{S}|$ is \emph{analytic} if it can be represented as $x\colon \Spa(k_x,k_x^+)\to\mathsf{S}$ such that $k(x)$ is not discrete. The subset $|\mathsf{S}|_\mr{a}\subseteq |\mathsf{S}|$ of analytic points is open and so corresponds to a unique open embedding of adic algebraic spaces $\mathsf{S}_\mr{a}\to \mathsf{S}$. We call $\mathsf{S}_\mr{a}$ the \emph{analytic locus} of $\mathsf{S}$. This defines a functor
     \begin{equation*}
    (-)_\mr{a}\colon \left\{\begin{matrix}\text{Adic algebraic}\\ \text{spaces}\end{matrix}\right\}^\ast \to  \left\{\begin{matrix}\text{Analytic adic}\\ \text{algebraic spaces}\end{matrix}\right\},
    \end{equation*}
    which is right adjoint to the inclusion of analytic adic algebraic spaces into the category of adic algebraic spaces with adic morphisms. 
\end{defin} 

\subsubsection*{Rigid locus and a result of Raynaud and Fujiwara--Kato} To begin, we observe there is a functor (recalling our conventions for formal schemes) 
\begin{equation*}
    (-)^\mr{ad}\colon \left\{\begin{matrix}\text{Formal}\\ \text{schemes}\end{matrix}\right\}\to\left\{\begin{matrix}\text{Adic}\\\text{spaces}\end{matrix}\right\},
\end{equation*}
uniquely characterized by preserving open embeddings/coverings and such that there is a functorial identification $\Spf(A)^\mr{ad}=\Spa(A)$ (e.g., combine \cite[Proposition 4.1]{HuberGen} with \cite{ZavyalovSheafiness}). We may compose the functors $(-)^\mr{ad}$ and $(-)_\mr{a}$ to obtain the \emph{rigid locus} functor 
\begin{equation*}
    (-)^\mr{rig} = (-)_\mr{a}\circ(-)^\mr{ad}\colon \left\{\begin{matrix}\text{Formal}\\ \text{schemes}\end{matrix}\right\}^\ast\to
    \left\{\begin{matrix}\text{Analytic}\\\text{adic spaces}\end{matrix}\right\},
\end{equation*}
where again the asterisk means restricting only to adic morphisms. If $\mf{S}=\Spf(A)$ where $(\pi)\subseteq A$ is as an ideal of definition, then $\mf{S}^\mr{rig}=\Spa(A[\nicefrac{1}{\pi}])$ (see \cite[Chapter II, \S A.4.(b)--\S A.4.(d)]{FujiwaraKato}).


The rigid locus functor extends to, and has quite pleasant properties on, the category of formal algebraic spaces locally of finite type over $\mf{S}$. To state this precisely, first recall that a~morphism $\mf{f}\colon \mf{X}'\to\mf{X}$ of formal schemes is an \emph{admissible blowup} if there exists a quasi-coherent, locally finitely generated, open ideal sheaf $\mc{J}\subseteq \mcO_\mf{X}$ such that $\mf{f}$ is final among adic and proper morphisms $\mf{g}\colon \mf{Y}\to\mf{X}$ with $\mf{g}^\ast\mc{J}\subseteq \mcO_\mf{Y}$ a Cartier divisor (see \cite[Chapter II, \S1.1.(a)-1.1.(b)]{FujiwaraKato} for a more explicit description). We say that $\mf{f}$ is an \emph{admissible modification} if it is adic and $\mf{f}^\mr{rig}$ is an isomorphism. If $\mf{X}'$ and $\mf{X}$ are coherent this condition is equivalent to the existence of a~decomposition $\pi_1\circ \mf{f}=\pi_2$ for a diagram of admissible blowups
\begin{equation*}
\mf{X}'\xleftarrow{\pi_1}\mf{Y}\xrightarrow{\pi_2}\mf{X}.
\end{equation*}
(see \cite[Chapter II, Corollary 2.1.5]{FujiwaraKato}). 

\begin{defin}  
    A morphism $\mf{f}\colon \mf{X}'\to\mf{X}$ of formal algebraic spaces is an \emph{admissible modification} if $\mf{X}'\times_\mf{X}\mf{Y}\to\mf{Y}$ is an admissible modification for any morphism $\mf{Y}\to\mf{X}$ from a formal scheme $\mf{Y}$.
\end{defin}

Denote by $W$ the class of admissible modifications of coherent formal algebraic spaces. This is left multiplicative (see \stacks[Section]{04VB}) and so the localization $(-)[W^{-1}]$ with respect to $W$ is well-behaved.

\begin{thm}[Raynaud, Fujiwara--Kato]\label{thm:rigid-locus-functor}  There exists a unique functor
    \[
        (-)^\mr{rig} \colon \cat{FAlgSp} \to\left\{\begin{matrix}\emph{Analytic adic}\\ \emph{algebraic spaces}\end{matrix}\right\}
    \]
    extending the rigid locus functor on formal schemes, and such that the natural map 
    \begin{equation*} 
        \mf{U}^\mr{rig}/\mf{R}^\mr{rig}\to (\mf{U}/\mf{R})^\mr{rig}
    \end{equation*}
    is an isomorphism for every \'etale equivalence relation $\mf{U}\rightrightarrows \mf{R}$. This functor sends $W$ to isomorphisms and for any coherent formal algebraic space $\mf{S}$ induces an equivalence of categories
    \begin{equation*}
        \cat{FAlgSp}^{\rm ft,adm}_{\mf{S}}[W^{-1}] \underset{\rm incl.}\isomto \cat{FAlgSp}^{\rm ft}_{\mf{S}}[W^{-1}] \underset{(-)^\mr{rig}}\isomto \cat{RigAlgSp}_{\mf{S}^\mr{rig}}^{\rm coh}.
    \end{equation*}
\end{thm}

\begin{warning} 
The finite-type hypotheses in the final statement of Theorem \ref{thm:rigid-locus-functor} are crucial. Let $K\subseteq L$ be a finite extension of non-archimedean fields with $\mcO_L$ not finite over $\mcO_K$ (see \cite[\S6.4.1]{BGR}). Write $L=K[x_1,\ldots,x_n]/(f_1,\ldots,f_m)$ with $f_i$ in $\mcO_K[x_1,\ldots,x_n]$, and set $A$ to be $\mcO_K\langle x_1,\ldots,x_n\rangle/(f_1,\ldots,f_m)$. Then, $\Spf(\mcO_L)$ is not finite type over $\Spf(\mcO_K)$ and
\begin{equation*}
    \Spf(\mcO_L)^\mr{rig}\simeq \Spa(L)\simeq \Spf(A)^\mr{rig}.
\end{equation*}
That said, there is no common admissible blowup of $\Spf(\mcO_L)$ and $\Spf(A)$, and so they are not isomorphic in the localization of $\cat{FSch}^\mr{coh}_{\mcO_K}$ with respect to admissible modifications. 
\end{warning}

We will often use the following terminology in the sequel.

\begin{term} 
    Let $\mathsf{X}$ (resp.\@ $\mathsf{f}\colon \mathsf{X}'\to \mathsf{X}$) be an object (resp.\@ morphism) of $\cat{RigAlgSp}_{\mf{S}^\mr{rig}}$.
    \begin{itemize}
        \item A \emph{model} of $\mathsf{X}$ is a formal algebraic space $\mf{X}$ locally of finite type over $\mf{S}$ an identification $\mf{X}^\mr{rig}\simeq \mathsf{X}$, which we always leave implicit.
        \item A \emph{model} of $\mathsf{f}$ is a morphism $\mf{f}\colon \mf{X}'\to\mf{X}$ of formal algebraic spaces locally of finite type over $\mf{S}$ together with an identification $\mathsf{f}\simeq \mf{f}^\mr{rig}$, which we always leave implicit.
        \item We say a model $\mf{X}$ (resp.\@ $\mf{f}$) of $\mathsf{X}$ (resp.\@ $\mathsf{f}$) is \emph{admissible} or \emph{coherent} if $\mf{X}$ (resp.\@ the source and target of $\mf{f}$ are both) admissible or coherent, respectively. 
    \end{itemize}
\end{term}

Finally, we recall a well-known result saying that essentially all reasonable properties $P$ of $\mf{f}$ are inherited by $\mf{f}^\mr{rig}$, and in many cases $\mathsf{f}$ satisfying $P$ is equivalent to having a model $\mf{f}$ satisfying $P$. 

\begin{thm}[Bosch--L\"utkebohmert--Raynaud]\label{thm:map-properties-and-generic-fiber} 
    Let $P$ be one of the following properties 
    \begin{multicols}{3}
    \begin{enumerate}[(i)]
        \item quasi-compact,
        \item open embedding,
        \item closed embedding,
        \item finite,
        \item separated,
        \item proper,
        \item flat,
        \item faithfully flat,
        \item \'etale.
    \end{enumerate}
    \end{multicols}
    \noindent If a morphism locally of finite type between formal schemes of type (N)/(V) satisfies $P$, then the induced map $\mf{f}^\mr{rig}$ satisfies $P$ (for cases (vii) and (viii) we assume the formal scheme is Jacobson). Moreover, except in case (ix), a finite type morphism $\mathsf{f}$ between coherent analytic adic spaces of type (N)/(V) satisfies $P$, then it has a coherent admissible formal model $\mf{f}$ which satisfies $P$.
\end{thm}

\subsubsection*{Specialization map} 

We now recall the existence of the specialization map, whose proof follows easily from bootstrapping the arguments in {\cite[Chapter II, Theorems 3.1.2 and Theorem 3.1.5]{FujiwaraKato}}. 

For a formal or adic rigid space $\mathscr{S}$ we let $\mathscr{S}_\Et$ denote the big \'etale topos. It is endowed with the usual structure sheaf given by the following formula
\begin{equation}\label{eq:structure-sheaf-def}
    \mcO_\ms{S}(\ms{U})= \varprojlim_{\ms{U}\to\ms{S}}\mcO_\ms{U}(\ms{U})
\end{equation}
with $\ms{U}\to\ms{S}$ ranges over morphisms to $\ms{S}$ from representable objects. When $\mathsf{S}$ is an adic algebraic space, one may similarly define the sheaf of rings $\mcO_\mathsf{S}^+$ 


\begin{prop}\label{prop:sp-omnibus} Let $\mf{S}$ be a formal algebraic space.
    \begin{enumerate}
    \item There exists a unique natural morphism
    \begin{equation*}
        \mr{sp}_{\mf{S}}\colon |\mf{S}^\mr{rig}|\to |\mf{S}|,
    \end{equation*}
    such that for $\mf{S}=\Spf(A)$ with ideal of definition $(\pi)\subseteq A$, one has 
    \begin{equation}\label{eq:sp-map-affine}
        \mr{sp}_{\mf{S}}\colon \Spa(A[\nicefrac{1}{\pi}])\to \Spf(A),\quad \nu\mapsto \{a\in A:\nu(a)<1\}.
    \end{equation}
    The map $\mr{sp}_{\mf{S}}$ is quasi-compact and closed, and is surjective if $\mf{S}$ is admissible. 
    \item There exists a unique morphism of locally topologically ringed spaces 
    \begin{equation*} 
    \mr{sp}_\mf{S}\colon (\mf{S}^\mr{rig},\mcO_{\mf{S}^\mr{rig}}^+)\to (\mf{S},\mcO_{\mf{S}})
    \end{equation*}
    functorial in a formal scheme $\mf{S}$ such that on global sections \eqref{eq:sp-map-affine} is the map $A\to A[\nicefrac{1}{\pi}]^\circ$, and which recovers (1) on the underlying topological space. 
    \item There exists a unique morphism of locally topologically ringed topoi
    \begin{equation*} \mr{sp}_\mf{S}\colon (\mf{S}^\mr{rig}_\Et,\mcO_{\mf{S}^\mr{rig}}^+)\to (\mf{S}_\Et,\mcO_{\mf{S}}),
    \end{equation*}
    functorial in a formal algebraic space $\mf{S}$  whose underlying morphism of sites associates sends $\mf{Y}\to\mf{X}$ to $\mf{Y}^\mr{rig}\to\mf{X}^\mr{rig}$, and which recovers (1) and (2) in the obvious sense.
    \end{enumerate}
    In all cases we call $\mr{sp}_\mf{S}$ the \emph{specialization morphism} associated to $\mf{S}$.
\end{prop}

\subsection{Gluing triples and Beauville--Laszlo gluing of algebraic spaces}
\label{ss:gluing-triple-category} 

We now define the category of gluing triples in their natural generality and state our main theorem, Theorem \ref{thm:main}.

\subsubsection*{Base setup}

We begin by fixing the base for which the objects we study will live over. In the sequel we shall use all the following notation without comment.

Let $S$ be a coherent algebraic space and let $S_0$ be a finitely presented closed subspace of $S$, cut out by a quasi-coherent ideal sheaf $\mc{I}\subseteq \cO_S$. We set 
\begin{itemize}
    \item $S^\circ=S\mysetminus S_0$, an open algebraic subspace of $S$,
    \item $S_n=V(\mc{I}^{n+1})$ for $n\geqslant 0$, a finitely presented closed subspace of $S$,
    \item $\wh{S}$ the completion of $S$ along $\mc{I}$, a formal algebraic space.
\end{itemize}
For an algebraic space $X$ over $S$ we shall consistently write $X^\circ$ instead of $X\times_S S^\circ$.

We shall always assume that $(S,S_0)$ is one of the following types:
\begin{itemize}
    \item {\emph{type (N):}} $S$ is locally Noetherian,
    \item {\emph{type (V):}} $S = \Spec(\mcO)$ for a complete rank one valuation ring $\mcO$, and $S_0 = V(\pi)$ for a~pseudouniformizer $\pi$ of $\mcO$.
\end{itemize}



\subsubsection*{Analytification} 

Suppose temporarily that $S$ is a scheme. Consider the following sequence of maps of locally ringed spaces
\begin{equation}\label{eq:rig-to-scheme}
    (\wh{S}^\mr{rig},\mcO_{\wh{S}^\mr{rig}})\xrightarrow{\mr{nat.}}(\wh{S}^\mr{rig},\mcO^+_{\wh{S}^\mr{rig}})\xrightarrow{\mr{sp}_{\wh{S}}}(\wh{S},\mcO_{\wh{S}})\xrightarrow{\mr{nat.}} (S,\mcO_S),
\end{equation}
where each $\mr{nat.}$ stands for the natural map. The composition of these maps then uniquely factorizes through the map $(S^\circ,\mcO_{S^\circ})\to (S,\mcO_S)$ as $\mc{I}\mcO_{\wh{S}^\mr{rig}}=\mcO_{\wh{S}^\mr{rig}}$. When $(S,S_0)=(\Spec(A),(\pi))$ then the map $(\wh{S}^\mr{rig},\mcO_{\wh{S}^\mr{rig}})\to (S^\circ,\mcO_{S^\circ})$ corresponds to the natural map of rings $A[\nicefrac{1}{\pi}]\to \wh{A}[\nicefrac{1}{\pi}]$. 

Using this map of locally ringed spaces, we can define the first instance analytification. 

\begin{defin}\label{def:analytification-scheme} 
    Suppose $S$ is a scheme and $U$ is a scheme locally of finite type over $S^\circ$. We define the \emph{analytification} of $U$ to be the adic space
    \begin{equation*}
        U^\mr{an}= U\times_{S^\circ}\wh{S}^\mr{rig},
    \end{equation*}
    where this fiber product is taken in the sense of \cite[Proposition 3.8]{HuberGen}.
\end{defin}

In order to extend this construction to algebraic spaces, we proceed in two steps. First, suppose that $S$ is a scheme, and let $U = W/R$ be an \'etale presentation of $U$. Then $R^\an \rightrightarrows W^\an$ is an \'etale equivalence relation in rigid spaces over $\wh{S}^\rig$, and the quotient space 
\[ 
    U^\an = W^\an / R^\an
\]
is independent of the chosen presentation (cf.\@ \cite[Lemma 2.2.1]{ConradTemkinII}). For $S$ general, we note that both algebraic spaces over $S^\circ$ and rigid algebraic spaces over $\wh{S}^\rig$ can be defined \'etale locally on $S$. In order to define $U^\an$ in general, we pick an \'etale presentation $S = S'/S''$ of the base algebraic space $S$. Then $(U\times_S S')^\an$ as defined above comes equipped with a descent datum to $S$ (or rather to $\wh{S}^\rig$), producing the desired rigid algebraic space $U^{\rm an}$.

\subsubsection*{Analytification and rigid locus of completion} 

Let $X$ be an algebraic space locally of finite type over $S$. One can construct a rigid algebraic space over $\wh{S}^{\rm rig}$ in two different ways: by taking the analytification $(X^\circ)^{\rm an}$ of $X^\circ = X\times_S S^\circ$, and by taking the rigid locus $\wh{X}^{\rm rig}$ of its formal completion $\wh{X}$ along $X_0 = X\times_S S_0$. These constructions yield a diagram of categories
\[ 
    \begin{tikzcd}
        \cat{AlgSp}^{\rm lft}_S \arrow[r] \arrow[d] & \cat{FAlgSp}^{\rm lft}_{\widehat{S}} \arrow[d] \\
        \cat{AlgSp}^{\rm lft}_{S^\circ} \arrow[r] & \cat{RigAlgSp}^{\rm lft}_{\widehat{S}^\circ,}
    \end{tikzcd}
\]
for which there exists a natural transformation (\emph{not} an {equivalence}) between its compositions.

\begin{prop}[{cf.\@ \cite[Proposition 1.9.6]{HuberEC}}]\label{prop:j-map}
    There is a unique map of rigid algebraic spaces $j_{X}\colon \widehat{X}^{\rm rig} \to (X^\circ)^{\rm an}$ over $\wh{S}^\mr{rig}$, functorial in $X$ and such that the following diagram of locally ringed spaces commutes when $X$ is a scheme
    \begin{equation*}
    \begin{tikzcd}[sep=large]
	{(\wh{X}^\mr{rig},\mcO_{\wh{X}^\mr{rig}})} & {((X^\circ)^\mr{an},\mcO_{(X^\circ)^\mr{an}})} \\
	{(\wh{X},\mcO_{\wh{X}})} & {(X,\mcO_X).}
	\arrow["{j_X}", from=1-1, to=1-2]
	\arrow[from=1-1, to=2-1]
	\arrow["{\eqref{eq:rig-to-scheme}}"{description}, from=1-1, to=2-2]
	\arrow[from=1-2, to=2-2]
	\arrow[from=2-1, to=2-2]
\end{tikzcd}
    \end{equation*}
    Moreover, the map $j_X$ is \'etale. If $X\to S$ is a separated, proper, or representable morphism then $j_X$ open embedding, isomorphism, or locally open embedding, respectively.
\end{prop}

\subsubsection*{Gluing triples} 

We now come to the category of gluing triples. 

\begin{defin} 
    A \emph{gluing triple} over $(S,S_0)$ is a triple $(U, \mf{X}, j)$ where
    \begin{itemize}
        \item $U$ is an algebraic locally of finite type over $S^\circ$,
        \item $\mf{X}$ is a formal algebraic space locally of finite type over $\wh{S}$,
        \item $j\colon \mf{X}^\mr{rig}\to U^\mr{an}$ is an \'etale morphism.
    \end{itemize}
    A morphism of gluing triples $(U_1, \mf{X}_1, j_1)\to (U, \mf{X}, j)$ is a pair of morphisms $\alpha\colon U_1\to U$ and $\beta\colon \mf{X}_1\to \mf{X}$ of algebraic spaces over $S^\circ$ and formal algebraic spaces over $\wh{S}$, respectively, such that $\alpha^\mr{an}\circ j_1=j\circ\beta^\mr{rig}$. We denote the category of gluing triples over $(S,S_0)$ by $\cat{Trip}_{(S,S_0)}$. 
\end{defin}

Informally, one can visualize a triple $(U,\mf{X},j)$ in the form of the following diagram (of ringed spaces/topoi), in which we interpret the leftmost and rightmost square as being `cartesian' 
\begin{equation} \label{eqn:triple-dgm}
    \begin{tikzcd}
        U\arrow[d] & U^{\rm an} \arrow[l] \arrow[d] & \mf{X}^{\rm rig} \arrow[l,swap,"j"] \arrow[r,"\rm sp"] \arrow[d] & \mf{X} \arrow[d] \\
        S^\circ & (S^\circ)^{\rm an} \arrow[l] \arrow[r,equal] & (\wh{S})^{\rm rig} \arrow[r] & \wh{S}.
    \end{tikzcd}
\end{equation}

We will often use the following terminology concerning gluing triples.

\begin{term} \label{term:propertyP}
    Let $(U,\mf{X},j)$ be a gluing triple over $(S,S_0)$.
    \begin{itemize}
        \item  Let $P$ be a property of morphisms of (formal) algebraic spaces. We say that a morphism $(U_1,\mf{X}_1,j_1)\to (U,\mf{X},j)$ is said to \emph{satisfy $P$} if both $U_1\to U$ and $\mf{X}_1\to \mf{X}$ satisfy $P$.
        \item A collection $\{(U_i,\mf{X}_i,j_{i})\to (U,\mf{X},j)\}$ of open embeddings (resp.\@ \'etale morphisms) is an \emph{open cover (resp.\@ \'etale cover)} if $\{U_i\to U\}$ and $\{\mf{X}_i\to\mf{X}\}$ are jointly surjective.
        \item We say that $(U,\mf{X},j)$ is \emph{separated} if $U\to S^\circ$ and $\mf{X}\to\wh{S}$ are separated maps and $j$ is an open embedding. We write $\cat{Trip}_{(S,S_0)}^\mr{sep}$ for the full subcategory of separated gluing triples.
    \end{itemize}
\end{term}

\begin{warning}
    We have defined the notion of separatedness for a triple in Terminology~\ref{term:propertyP} in two different ways: implicitly in the first point with $P$ being `separated', and in the third point. These notions are not equivalent: if $(U, \mf{X}, j)$ is separated, then both $U$ and $\mf{X}$ are separated, but not conversely. We shall always use the more restrictive notion.
\end{warning}

The category $\cat{Trip}_{(S,S_0)}$ has a final object given by the triple $(S^\circ,\wh{S},\mathrm{id})$. It furthermore admits all fiber products, computed in the obvious way:
\begin{equation*}
    (U_1,\mf{X}_1,j_1)\times_{(U_2,\mf{X}_2,j_2)}(U_3,\mf{X}_3,j_3)=(U_1\times_{U_2}U_3,\mf{X}_1\times_{\mf{X}_2}\mf{X}_3,j_1\times j_3).
\end{equation*}
Moreover, if $(T,T_0)\to (S,S_0)$ is a morphism, with $|T_0|=|T\times S_0|\subseteq |T|$,
then there is a natural base change functor $\cat{Trip}_{(S,S_0)}\to \cat{Trip}_{(T,T_0)}$ computed in the same way. Both of these operations preserves the full subcategories of separated triples.

We now summarize the basic descent properties of the category of gluing triples below. Before we do so, let us establish the following terminology. By an \emph{\'etale gluing datum} in $\cat{Trip}_{(S,S_0)}$ we mean an index set $\Sigma$ and collections of triples 
\begin{equation*} 
    \{T_\sigma=(U_\sigma,\mf{X}_\sigma,j_
    \sigma)\}_{\sigma\in\Sigma},\qquad \{T_{\sigma\tau}=(U_{\sigma\tau},\mf{X}_{\sigma\tau},j_{\sigma\tau})\}_{\sigma,\tau\in\Sigma},
\end{equation*}
with $T_{\sigma\tau}=T_{\tau\sigma}$, together with \'etale morphisms $\phi^\sigma_{\sigma\tau}\colon T_{\sigma\tau}\to T_{\sigma}$. We say that an \'etale gluing datum is \emph{separated} if for all $\sigma,\tau\in\Sigma$ the map 
\begin{equation*} 
    (\phi^\sigma_{\sigma\tau},\phi^\tau_{\sigma\tau})\colon T_{\sigma\tau}\to T_\sigma\times T_\tau
\end{equation*}
is a closed embedding. We say that an {open gluing datum} is \emph{effective} if the natural diagram in $\cat{Trip}_{(S,S_0)}$ it forms admits a colimit. The following statement is obvious.

\begin{prop}\label{prop:trip-descent} 
    Every \'etale gluing datum in $\cat{Trip}_{(S,S_0)}$ is effective. Every separated \'etale gluing datum in $\cat{Trip}_{(S,S_0)}^\mr{sep}$ is effective.
\end{prop}

\begin{cor} 
    For an \'etale map $S'\to S$, set $S'_0=S'\times_S S_0$. Then, the associations of $\cat{Trip}_{(S',S'_0)}$ and $\cat{Trip}_{(S',S'_0)}^\mr{sep}$ to $S'\to S$ are stacks for the \'etale topology.
\end{cor}

\subsubsection*{Beauville--Laszlo Gluing of algebraic spaces} 

To state our main theorem we first observe that Proposition \ref{prop:j-map} allows us to realize algebraic spaces over $S$ as gluing triples.

\begin{defin}\label{def:t-functor} Let $X$ be an algebraic space locally of finite type over $S$. The \emph{(gluing) triple associated} to $X$ is defined as follows:
\begin{equation*}
    \mct(X)= (X^\circ,\wh{X},j_X).
\end{equation*}
This association gives rise to functors
    \begin{equation*}
        \mct \colon \cat{AlgSp}^{\rm lft}_S \to \cat{Trip}_{(S, S_0)},\qquad  \mct \colon \cat{AlgSp}^{\rm lft,sep}_S \to \cat{Trip}^\mr{sep}_{(S, S_0)}.
    \end{equation*}
\end{defin}

Again, one can visualize the triple $\mct(X) = (X^\circ, \widehat{X}, j_X)$ in the form of a diagram, in which the diagram \eqref{eqn:triple-dgm} features as the three back faces of the prism. 
\[
    \begin{tikzcd}[column sep=small, row sep=small]
        & (X^\circ)^{\rm an}\ar[dddd]\arrow[dl] & & \wh{X}{}^{\rm rig}\arrow[dddd]\arrow[ll,"j_X",swap] \arrow[dr] \\
      X^\circ\arrow[dddd]\ar[drr,crossing over] &   & &   & \wh{X}\arrow[dddd]\arrow[dll,crossing over] \\
        &   & X & \\
        \\
        & (S^\circ)^{\rm an}\arrow[dl] \arrow[rr,equal] & & (\widehat{S})^{\rm rig}\arrow[dr] \\
      S^\circ\arrow[drr] &   & &   & \widehat{S}\arrow[dll] \\
        &   &S \arrow[from=uuuu,crossing over]& \\
    \end{tikzcd}
\]

We come now to our main theorem, whose proof will occupy the entirety of \S\ref{s:proof-of-gluing-theorem}. We preface this by saying that an algebraic space $X$ locally of finite type over $S$ is the \emph{gluing} of a gluing triple $(U,\mf{X},j)$ if $\mct(X)\simeq(U,\mf{X},j)$. Informally, this means that the square below is a pushout
\begin{equation*}
    \begin{tikzcd}[sep=large]
	& {\mf{X}^\mr{rig}} & {\mf{X}} \\
	{U^\mr{an}} & U & X
	\arrow["{\mr{sp}_\mf{X}}", from=1-2, to=1-3]
	\arrow["j"', from=1-2, to=2-1]
	\arrow[from=1-2, to=2-2]
	\arrow[from=1-3, to=2-3]
	\arrow["{\mr{nat.}}"', from=2-1, to=2-2]
	\arrow[from=2-2, to=2-3]
	\arrow["\ulcorner"{anchor=center, pos=0.125, rotate=180}, draw=none, from=2-3, to=1-2]
\end{tikzcd}
\end{equation*}
Beauville--Laszlo Gluing of algebraic spaces say such a gluing always exists if $S$ is a $G$-space.

\begin{defin}\label{def:G-space} 
    An algebraic space $S$ is called a \emph{$G$-space} if for each finite type point $s$ of $S$ (see \stacks[Section]{06EE}) there exists an \'etale morphism $\varphi\colon S'\to S$ with $S'$ a scheme, and a closed point $s'$ of $S'$ such such that $\varphi(s')=s$ and $\mcO_{S',s'}$ is a $G$-ring in the sense of \stacks[Definition]{07GH}.
\end{defin}

\begin{example} 
If $S$ is an algebraic space which admits an \'etale cover $\bigsqcup_i \Spec(A_i)\to S$ where each $A_i$ is an excellent ring (in the sense of \stacks[Section]{07QS}), then $S$ is a $G$-space.
\end{example}

\begin{thm}[Beauville--Laszlo Gluing of algebraic spaces]\label{thm:main} Let $(S,S_0)$ be a pair of type (N)/(V). Then, the functor
        \begin{equation*}
            \mct \colon \cat{AlgSp}^{\mr{lft,sep}}_S \to \cat{Trip}^\mr{sep}_{(S, S_0)}
        \end{equation*}
from Definition \ref{def:t-functor} is fully faithful. Furthermore, if $S$ is a $G$-space then $\mct$ is an equivalence of categories with quasi-inverse
    \begin{equation*}
        \mcg\colon \cat{Trip}^\mr{sep}_{(S,S_0)}\to \cat{AlgSp}_S^{\mr{lft},\mr{sep}},\quad (U,\mf{X},j)\mapsto \left(Y\mapsto \Hom_{\cat{Trip}_{(S,S_0)}}\left(\mct(Y),(U,\mf{X},j)\right)\right).
    \end{equation*}
We sometimes call the functor $\mcg$ the \emph{gluing functor}.
\end{thm}

\subsubsection*{Properties of triples and morphisms of triples}  Finally we give an omnibus result that shows that the functor $\mct$ reflects and preserves most frequently-encountered properties of morphisms. We defer its proof to \S\ref{ss:proof-of-morphism-properties}, as it relies on later material.


\begin{prop}\label{prop:triples-morphisms-properties} 
    Let $f\colon Y\to X$ be a morphism of algebraic spaces separated and locally of finite type over $S$, and let $P$ be one of the following properties:
    \begin{multicols}{3}
    \begin{enumerate}[i)]
        \item quasi-compact,
        \item surjective,
        \item open embedding,
        \item closed embedding,
        \item (locally) quasi-finite,
        \item isomorphism,
        \item finite,
        \item separated,
        \item flat,
        \item \'etale,
        \item smooth.
    \end{enumerate}
    \end{multicols}
    \noindent Then, the map $f$ has property $P$ if and only if $\mct(f)$ does.
\end{prop}

\section{Examples and further results}\label{s:examples} 

In this section we illustrate Theorem \ref{thm:main} by giving several concrete examples of gluing. When applicable, we also remark on natural corollaries of Beauville--Laszlo Gluing of algebraic spaces that the example more generally suggests.

\subsection{Gluing produces finite type outputs} We begin by showing how simple gluing triples can give rise to quite exotic-looking schemes.

\begin{example}[Two formal disks glued to the affine line]\label{eg:two-disks-one-A1} Consider the gluing triple 
\begin{equation*}(\mathbb{A}^1_{\mathbb{Q}_p},\widehat{\mathbb{A}}^1_{\mathbb{Z}_p}\sqcup \widehat{\mathbb{A}}^1_{\mathbb{Z}_p},j),
\end{equation*}
where $j$ is the natural embedding  
\begin{equation*}
    (\widehat{\mathbb{A}}^1_{\mathbb{Z}_p}\sqcup \widehat{\mathbb{A}}^1_{\mathbb{Z}_p})^\mr{rig}\isomto \left\{x\in\mathbb{A}^{1,\mr{an}}_{\mathbb{Q}_p}:|x|\leqslant 1\right\}\sqcup \left\{x\in\mathbb{A}^{1,\mr{an}}_{\mathbb{Q}_p}:|x-\nicefrac{1}{p}|\leqslant 1\right\}\subseteq \mathbb{A}^{1,\mr{an}}_{\mathbb{Q}_p}.
\end{equation*}
In this case one may show that $\mcg(\mathbb{A}^1_{\mathbb{Q}_p},\widehat{\mathbb{A}}^1_{\mathbb{Z}_p}\sqcup \widehat{\mathbb{A}}^1_{\mathbb{Z}_p},j)\simeq \Spec(A)$ where 
\begin{equation}\label{eq:two-disks-one-A1}
    A=\left\{f(x)\in\mathbb{Z}_p[x]: f(x+\nicefrac{1}{p})\in\mathbb{Z}_p[x]\right\}.
\end{equation}
In particular, the special fiber of the resulting affine scheme $X = \Spec(A)$ is isomorphic to the disjoint union of two copies of $\mathbb{A}^1_{\mathbb{F}_p}$.

The reason this example is worth remarking on, is that it is not even clear whether the ring $A$ from \eqref{eq:two-disks-one-A1} is a finitely generated $\mathbb{Z}_p$-algebra. That said, one may check by hand that it is generated by the elements 
\[ 
    \alpha = x(1-px)^2 \qquad \beta = px^2(1-px), \qquad \gamma = px 
\]
subject to the relations
\[ 
    p\alpha = \gamma(1-\gamma)^2, \qquad
    p\beta = \gamma^2(1-\gamma), \qquad
    \gamma\alpha = (1-\gamma)\beta.
\]

\end{example}

In fact, one has that such a gluing is finite type quite generally as the following result shows. We delay its proof until \S\ref{ss:proof-of-fg} due to its length.

\begin{prop}\label{prop:affine-gluing-base-case} Let $R$ be a Noetherian ring and $(\pi)\subseteq R$ an ideal. We further let
\begin{itemize} 
\item $A$ be a finitely generated $R[\nicefrac{1}{\pi}]$-algebra, 
\item $B$ be a topologically finitely generated $\wh{R}$-algebra, 
\item and $j^*\colon A\to C=B[\nicefrac{1}{\pi}]$ be a map of $R[\nicefrac{1}{\pi}]$-algebras with dense image and for which the induced map $\Spa(C) \to \Spec(A)^{\rm an}$ is an open immersion. 
\end{itemize}
Then the $R$-algebra $D$ defined as the pull-back   
\begin{equation*}
    \begin{tikzcd}
	D & B \\
	A & C,
	\arrow[from=1-1, to=1-2]
	\arrow[from=1-1, to=2-1]
	\arrow["\lrcorner"{anchor=center, pos=0.125}, draw=none, from=1-1, to=2-2]
	\arrow[from=1-2, to=2-2]
	\arrow["j"', from=2-1, to=2-2]
\end{tikzcd}
\end{equation*}
is finitely generated, and satisfies $D[\nicefrac{1}{\pi}]=A$ and $\widehat{D} = B$.
\end{prop}

\subsection{Characterization of (affine) schematic gluing triples} 

It is natural to ask how to recognize the `(affine) schematic separated gluing triples', i.e., the essential image of separated (affine) schemes under the functor $\mct$. One might guess the answer is that separated gluing triple $(U,\mf{X},j)$ is (affine) schematic if $U$ and $\mf{X}$ are (affine) formal schemes. The following example shows this is not right: one can glue two affine objects together to obtain a non-affine object.

\begin{example}\label{eg:non-affine-glued-affine} 
Let $\mathbb{A}^1_{\mathbb{Z}_p}=\Spec(\mathbb{Z}_p[x])$. Then, $X=\mathbb{A}^1_{\mathbb{Z}_p}\mysetminus V(p,x)$ satisfies $\mct(X)=(\mathbb{A}^1_{\mathbb{Z}_p},\wh{\bb{G}}_{m,\mathbb{Z}_p},j)$ where $j$ is the natural inclusion of the unit circle $\wh{\mathbb{G}}_{m,\mathbb{Z}_p}^\mr{rig}\simeq \{x\in\mathbb{A}^{1,\mr{an}}_{\mathbb{Q}_p}:|x|=1\}$ into $\mathbb{A}^{1,\mr{an}}_{\mathbb{Q}_p}$.
\end{example}

More seriously, it is not even true that the result of gluing a formal scheme to a scheme along a rigid analytic open is \emph{even a scheme} as the following example shows.

\begin{example}[Matsumoto]\label{eg:Matsumoto} 
    Matsumoto \cite[Example~5.2]{Matsumoto} observed that a K3 surface over a~discretely valued field might have good reduction as an algebraic space but not as a scheme. More precisely, for every prime $p\geqslant 7$ he constructed a smooth and proper algebraic space $X$ over $\mathbb{Z}_p$ whose fibers $X_{\mathbb{Q}_p}$ and $X_{\mathbb{F}_p}$ are (projective) K3 surfaces, but such that $X_{\mathbb{Q}_p}$ does not admit a smooth proper scheme model over $\mathbb{Z}_p$. Moreover, this property persists even after passing to $\mcO_K$ for a finite extension $K$ of $\mathbb{Q}_p$. The generic fiber $X_{\mathbb{Q}_p}$ is a double cover of $\mathbb{P}^2_{\mathbb{Q}_p}$ ramified along a sextic and has Picard number one. What stops $X$ from being projective itself is that no extension of an ample line bundle on the generic fiber is ample on the special fiber, and no ample line bundle on the special fiber lifts to a line bundle on $X$. 

    The triple $\mct(X) = (X_{\mathbb{Q}_p},\wh{X},j_X)$ is separated ($j_X$ is even an isomorphism), the algebraic space $X_{\mathbb{Q}_p}$ is a projective scheme, and the formal algebraic space $\wh{X}$ is a proper formal scheme. In fact, by a simple deformation theory argument one may see that $\wh{X}_n$ is projective for all $n\geqslant 0$, though $\wh{X}$ itself admits no line bundle whose restriction to $\wh{X}_0 = X_{\mathbb{F}_p}$ is ample. Matsumoto's observations combined with Theorem~\ref{thm:main} imply that there does not exist a scheme $X'$ over $\mathbb{Z}_p$ with $\mct(X') = \mct(X)$.    
\end{example}

The missing hypothesis in these examples is one involving the topology of rings: for a ring $A$ and element $\pi$, the natural map $j_{\Spec(A)}^\ast\colon A[\nicefrac{1}{\pi}]\to \wh{A}[\nicefrac{1}{\pi}]$ has dense image. This explains the phenomenon witnessed in Example \ref{eg:non-affine-glued-affine} as $j^\ast\colon \mathbb{Q}_p[x]\to \mathbb{Q}_p\langle x^{\pm1 }\rangle$ does not have dense image.

\begin{defin} \label{def:affine-triples}
    Suppose $(S,S_0)=(\Spec(R),V(\pi))$ where $R$ is a Noetherian ring. A gluing triple $(U,\mf{X},j)$ over $(S, S_0)$ is \emph{affine} if the following conditions are satisfied:
    \begin{enumerate}[(a)]
        \item $U = \Spec(A)$ is affine,
        \item $\mf{X} = \Spf(\widehat{B})$ is affine, 
        \item $j\colon \Spa(\wh{B}[\nicefrac{1}{\pi}])\to \Spec(A)^{\rm an}$ is an open immersion,
        \item and $j^*\colon A\to B$ has ($\pi$-adically) dense image.
    \end{enumerate}
\end{defin}

\begin{prop}\label{prop:affine-triples} 
    Suppose $(S,S_0)=(\Spec(R),V(\pi))$ and $R$ is Noetherian. The functor $\mct$ induces an equivalence between the category of finite type affine schemes over $\Spec(R)$ and that of affine gluing triples over $(S,S_0)$. 
\end{prop}

\begin{proof} 
The fully faithfulness of $\mct$ follows Theorem \ref{thm:main} (or more precisely  Corollary \ref{cor:fully-faithfulness-schemes} below), and the essential surjectivity follows from Proposition \ref{prop:affine-gluing-base-case}.
\end{proof}

\begin{defin} \label{def:schematic-triple}
    A gluing triple  $(U,\mf{X},j)$ is \emph{schematic} if it admits an open cover by affine gluing triples $(U_i, \mf{X}_i, j_i)$. 
\end{defin}

\begin{cor}\label{cor:schematic-triple}
Let $(S,S_0)$ be as in Proposition \ref{prop:affine-triples}. Then $\mct$ induces an equivalence between the category of schemes locally of finite type over $S$ and that of schematic gluing triples over $(S,S_0)$.
\end{cor}

\begin{remark} 
A characterization similar to that in Proposition~\ref{prop:affine-triples} appears in \cite[Remark 1.6]{BLNeron} in the case of Type (V) and with the restriction that the triples have reduced special fiber. This latter condition is quite restrictive and considerably simplifies the proof of the important Proposition \ref{prop:affine-gluing-base-case} as, with notation in that proposition, $D$ is then the elements of $A$ with $\pi$-adic absolute value bounded by $1$.
\end{remark}

Finally, we observe that Theorem \ref{thm:main} (together with (\stacks[Proposition]{03XX}) gives a more practical criterion for proving that a separated gluing triple is schematic.

\begin{prop} 
    If $S$ is a $G$-space a separated gluing triple $(U,\mf{X},j)$ over $(S,S_0)$ is schematic if it admits a separated and locally quasi-finite map to a separated schematic gluing triple.
\end{prop}

\subsection{\texorpdfstring{Good reduction of algebraic spaces over $K$ vs.\@ $\wh{K}$}{Good reduction of algebraic spaces over K vs. K hat}} 

Examples similar to Example \ref{eg:Matsumoto} arise naturally in simple questions regarding good reduction of a smooth proper $\mathbb{Q}$-scheme $X$:
\vspace*{2 pt}
\begin{quotation}
    \emph{If $X_{\mathbb{Q}_p}$ has good reduction (i.e., admits a smooth proper model over $\mathbb{Z}_p$), does $X$ have good reduction at $p$ (i.e., admit a smooth proper model over $\mathbb{Z}_{(p)}$)?}
\end{quotation}
\vspace*{2 pt}

\noindent The following example shows this question has a negative answer in general.

\begin{example}[Poonen] \label{eg:Poonen} In \cite[Warning 3.5.79]{Poonen} there is constructed a smooth quartic surface $X\subseteq\mathbb{P}^3_\mathbb{Q}$ which \emph{does not} admit a smooth proper model over $\mathbb{Z}_{(p)}$, but such that $X_{\mathbb{Q}_p}$ \emph{does} admit a~smooth proper model $\mc{X}_{\mathbb{Z}_p}$ over $\mathbb{Z}_p$. In particular, $\mc{X}= \mcg(X_{\mathbb{Q}},\wh{\mc{X}}_{\mathbb{Z}_p},j_{\mc{X
}_{\mathbb{Z}_p}})$ is another (see Example \ref{eg:Matsumoto}) non-representable algebraic space with representable components.
\end{example}

As is implicit in Example \ref{eg:Poonen} (see also \cite[Remark 3.5.80]{Poonen}), Theorem \ref{thm:main}, or the earlier results of \cite{ArtinII}, show that the answer to the above-posed question is no, but only because we insisted on having a model which is a scheme. More generally, Theorem \ref{thm:main} implies the following result (which is not directly deducible from \cite{ArtinII}). 

\begin{prop}\label{prop:models-over-K-vs-hat-K} 
    Let $K$ be a discrete valuation field with excellent valuation ring $\mcO$. Let
    \begin{itemize} 
    \item $X$ be an algebraic space separated and locally of finite type over $K$,
    \item and $\mc{X}_{\wh{\mcO}}$ be a separated and locally of finite type algebraic space model of $X_{\wh{K}}$ over $\wh{\mcO}$.
    \end{itemize}
    Then, there exists a unique separated and locally of finite type algebraic space model $\mc{X}$ over $\mcO$ of $X$ whose base change to $\wh{\mcO}$ is isomorphic to $\mc{X}_{\wh{\mcO}}$.
\end{prop}

\subsection{N\'eron models and triples} 

It was already observed in \cite{BLNeron} that a theory like that of gluing triples is useful to understand N\'eron models of abelian varieties. This usefulness can also be seen to holds if one replaces abelian varieties by tori, as the following example illustrates. 

\begin{example}\label{eg:Neron-model-Gm} 
Let $K$ be a discrete valuation field with valuation ring $\mcO$ and uniformizer $\pi$. Set
\begin{equation*}
    U_K= \mathbb{G}_{m,K},\qquad \mf{X}_\mcO= \bigsqcup_{n\in\mathbb{Z}}\wh{\mathbb{G}}_{m,\wh{\mcO}}\simeq \wh{\bb{G}}_{m,\wh{\mcO}}\times \underline{\mathbb{Z}}.
\end{equation*}
For an element $\omega$ of $\mcO$, define the open embedding $j_\omega\colon \mf{X}_\mcO^\mr{rig}\hookrightarrow U_K^\mr{an}$ on the $n^\text{th}$ component via 
\begin{equation*}
    \wh{\mathbb{G}}_{m,\wh{\mcO}}^\mr{rig}\isomto \{x\in\mathbb{G}_{m,\wh{K}}^\mr{an}:|x|=|\omega^n|\}\subseteq \mathbb{G}_{m,\wh{K}}^\mr{an}.
\end{equation*}
It is clear that $(U_K,\mf{X}_\mcO,j_\omega)$ defines a group object in the category separated gluing triples over $(\Spec(\mcO),V(\pi))$. Moreover, taking $\omega$ to be $\pi$, its gluing gives the N\'eron model of $\bb{G}_{m,K}$ over $\mcO$.
\end{example}

Example \ref{eg:Neron-model-Gm} nicely illustrates why N\'eron models do not commute with ramified base change. Namely, if $\mcO\subseteq \mcO'$ is an extension then it is clear that the base change of the triple $(U_K,\mf{X}_\mcO,j_\pi)$ along the map $(\Spec(\mcO'),V(\pi))\to (\Spec(\mcO),V(\pi))$ is $(U_{K'},\mf{X}_{\mcO'},j_\pi)$. But, the N\'eron model for $\mathbb{G}_{m,K'}$ over $\mcO'$ is $(U_{K'},\mf{X}_{\mcO'},j_{\pi'})$ where $\pi'$ is a uniformizer of $\mcO'$. It is easy to check that 
\begin{equation*}
     (U_{K'},\mf{X}_{\mcO'},j_\pi)\simeq  (U_{K'},\mf{X}_{\mcO'},j_{\pi'})\,\,\Longleftrightarrow\,\, |\pi|=|\pi'|\,\,\Longleftrightarrow \,\,\mcO\subseteq\mcO'\text{ is unramified},
\end{equation*}
as predicted by the behavior of N\'eron models under base change.

We find it an interesting question whether one can extend Example \ref{eg:Neron-model-Gm} to arbitrary tori.

\begin{question} 
Can one extend Example \ref{eg:Neron-model-Gm} to construct N\'eron models of all tori, and their finite-type and connected analogues as studied in Bruhat--Tits theory (see \cite[Appendix B]{KalethaPrasad})?
\end{question}

\begin{remark}\label{rem:BLNeron} 
We lastly remark that Theorem \ref{thm:main} also conceptually simplifies the construction in \cite{BLNeron}. In our terminology and with notation as in Example \ref{eg:Neron-model-Gm}, in \cite{BLNeron} they construct for an abelian variety $A$ over $K$ a gluing triple $(A,\mf{A},j)$ so that if $(A,\mf{A},j)=\mct(\mc{A})$ for some group $\mcO$-scheme $\mc{A}$, then $\mc{A}$ is the N\'eron model of $A$. Without Theorem \ref{thm:main}, the construction of $\mc{A}$ is quite involved. But, using Theorem \ref{thm:main}, $\mcg(A,\mf{A},j)$ is a separated group algebraic $\mcO$-space which is automatically representable by \cite[Theorem 3.3.1]{RaynaudSpecialisation}.
\end{remark}

\subsection{Algebraizability of a formal scheme is detected by its rigid locus} 

Finally, we point out that using Theorem \ref{thm:main} one can show that algebraizability of a formal scheme is detected by its rigid locus if one expands the notion of algebraizability to include algebraic spaces. The following example illustrates this idea.

\begin{example}[Tate uniformization]\label{eg:DR}  Let $K$ be a non-archimedean field with residue field $k$. Let $E$ an elliptic curve over $K$ with split multiplicative reduction and let $\mc{E}$ be its minimal Weierstrass model over $\mcO$ (e.g., \cite[\S VII.1]{Silverman}). If $k$ is the residue field of $K$, then $\mc{E}\simeq \mathbb{P}^1_k/(0\sim\infty)$ which admits a $\mathbb{Z}$-covering space $f\colon \mc{X}_k\to\mc{E}_k$ by an infinite chain of copies of $\mathbb{P}^1_k$ glued pole-to-pole. Let $\mathfrak{X}\to\mf{E}$ be the unique \'etale deformation of $f$. One has an isomorphism $j\colon \mf{X}^{\rm rig} \isomto \mathbb{G}^\mr{an}_{m,K}$, and hence by Theorem \ref{thm:main} there exists an algebraic space $\mc{X}$ locally of finite type over $\cO_K$ with generic fiber $\mathbb{G}_{m,K}$ and formal completion $\mf{X}$. Using Corollary~\ref{cor:schematic-triple} it is not hard to see that $\mc{X}$ is in fact a scheme. Such a scheme has been explicitly described in \cite[\S IV.1]{DeligneRapoport}.
\end{example}

To explain the general application of Theorem \ref{thm:main}, we first give some notation. Fix a separated locally of finite type formal algebraic space $\mf{X}$ over $\wh{S}$. Set 
\begin{itemize} 
    \item $\mc{A}(\mf{X})$ to be the category of \emph{algebraizations} of $\mf{X}$, i.e., pairs $(X,\iota)$ where $X$ is an algebraic space separated and locally of finite type over $S$ and $\iota\colon \mf{X}\isomto \wh{X}$ an isomorphism,
    \item $\mc{E}(\mf{X}_\eta)$ to be the category of pairs $(U,j)$ where $U$ is an algebraic space separated and locally of finite type over $S^\circ$ and $j\colon \mf{X}^\mr{rig}\to U^\mr{an}$ an open embedding.
    \end{itemize}
The following proposition is an immediate corollary of, and largely equivalent to, Theorem \ref{thm:main}.

\begin{prop}\label{prop:A-E-equiv} Suppose $S$ is a $G$-space. Then, the functor
    \begin{equation*}
        \mc{A}(\mf{X})\to \mc{E}(\mf{X}_\eta),\quad (X,\iota)\mapsto (\wh{X},j_X\circ \iota^\mr{rig})
    \end{equation*}
    is an equivalence with quasi-inverse given by sending $(U,j)$ to $\mcg(U,\mf{X},j)$. 
\end{prop}

Let us say that a separated rigid algebraic space $\mathsf{X}$ over $\wh{S}^\mr{rig}$ is \emph{weakly algebraizable} if there exists an algebraic space $U$  separated and locally of finite type over $S^\circ$ and an open embedding $\mathsf{X}\hookrightarrow U^\mr{an}$. This is a better behaved notion than that of literal algebraizability, as $\mathsf{S}$ can never be algebraizable if it is quasi-compact and not proper over $\wh{S}^\mr{rig}$. 

\begin{cor} Suppose that $S$ is a $G$-space and $\mf{X}$ is a formal algebraic space separated and locally of finite type over $\wh{S}$. Then, $\mf{X}$ is algebraizable if and only if $\mf{X}^\mr{rig}$ is weakly algebraizable.
\end{cor}

\subsection{Gluing torsors} Theorem \ref{thm:main} implies that one can always glue a $G$-torsor over $\mct(X)$ to one over $X$, in a way which we now make precise. Throughout this subsection let us fix $G$ to be a separated, flat, and locally of finite type group algebraic space over $S$.

Let $X$ be an algebraic space over $S$ and $\mf{X}$ a formal algebraic space over $\wh{S}$. Recall that:
\begin{itemize}
\item a $G$-torsor on $X$ is a locally non-empty sheaf $Q$ on the fppf site of $X$ together with a~simply transitive action $G\times Q\to Q$,
\item a $\wh{G}$-torsor on $\mf{X}$ is a locally non-empty sheaf $\mf{Q}$ on the adic flat site of $\mf{X}$ together with a~simply transitive action $\wh{G}\times\mf{Q}\to \mf{Q}$.
\end{itemize}
Recall here that a sheaf $\mathcal{F}$ on a site $\ms{S}$ is \emph{locally non-empty} if for every object $S$ of $\ms{S}$ there exists a cover $\{S_i\to S\}$ such that $\mathcal{F}(S_i)$ is non-empty for all $i$. If there is a group sheaf $\mathcal{G}$ on $\ms{S}$, then an action $\mathcal{G}\times\mathcal{F}\to\mathcal{F}$ is \emph{simply transitive} if the induced map 
\begin{equation*} 
\mathcal{G}\times\mathcal{F}\to\mathcal{F}\times\mathcal{F},\qquad (g,f)\mapsto (gf,f),
\end{equation*}
is an isomorphism. Finally, we recall that the \emph{adic flat site} on $\mf{X}$ consists of all formal schemes over $\mf{X}$ endowed with the topology generated by Zariski open covers and adically faithfully flat maps $\mf{Y}'\to\mf{Y}$ (see \cite[Chapter I, \S4.8.(c)]{FujiwaraKato}). 

We denote these categories of torsors by $\cat{Tors}_G(X)$ and $\cat{Tors}_{\wh{G}}(\mf{X})$, respectively. By \stacks[Lemma]{04SK} any object of $\cat{Tors}_G(X)$ and $\cat{Tors}_{\wh{G}}(\mf{X})$ is represented by a (formal) algebraic space separated and locally of finite type over $X$ and $\mf{X}$, respectively. Thus, we treat torsors implicitly as such algebraic spaces in the sequel.

\begin{defin} 
    Let $(U,\mf{X},j)$ be a gluing triple over $S$. We then define the category $\cat{Tors}_{G}(U,\mf{X},j)$ of \emph{${G}$-torsors over $(U,\mf{X},j)$} to consist of triples $(\mcQ,\mf{Q},\iota)$ where
    \begin{itemize}
        \item $\mcQ$ is an object of $\cat{Tors}_{G}(U)$,
        \item $\mf{{Q}}$ is an object of $\cat{Tors}_{\wh{G}}(\mf{X})$,
        \item $\iota$ is a morphism of analytic adic algebraic spaces $\mf{Q}^\mr{rig}\to \mcQ^\mr{an}\times_{U^\mr{an}}\mf{X}^\mr{rig}$ equivariant for the map of group rigid algebraic spaces ${\wh{G}}^\mr{rig}\to ({G}^\circ)^\mr{an}$ over $\mf{X}^\mr{rig}$.
    \end{itemize}
\end{defin}

Suppose now that $X$ is an algebraic space separated and locally of finite type over $S$. Then, there is a functor
\begin{equation*}
    \tau\colon \cat{Tors}_{G}(X)\to \cat{Tors}_{G}(\mct(X)),\quad Q\mapsto \tau(Q)=(Q^\circ,\wh{Q},\iota_Q),
\end{equation*}
which is well-defined as the natural open embedding $\iota_Q\colon \wh{Q}^\mr{rig}\to (Q^\circ)^{\mr{an}}\times_{(X^\circ)^\mr{an}}\wh{X}^\mr{rig}$ is evidently equivariant for the map of group rigid-algebraic spaces $\wh{G}^\mr{rig}\to (G^\circ)^\mr{an}$.

The following is an immediate consequence of Theorem \ref{thm:main}.

\begin{prop}\label{prop:torsor-gluing} 
    Suppose that $S$ is a $G$-space and let $X$ be a algebraic space separated and locally of finite type over $S$. Then,
    \begin{equation*}
        \tau\colon \cat{Tors}_{G}(X)\to \cat{Tors}_{G}(\mct(X)),\quad Q\mapsto \tau(Q)=(Q^\circ,\wh{Q},\iota_Q).
    \end{equation*}
    is an equivalence of categories.
\end{prop}

\subsection{Gluing finite \'etale covers and an integral Riemann existence theorem} 

Finally, we use gluing triples to give a finer understanding of the finite \'etale covers of an algebraic space.

\begin{defin} 
    Let $(U,\mf{X},j)$ be a separated gluing triple over $(S,S_0)$. Then, we define the category $\FEt(U,\mf{X},j)$ of \emph{finite \'etale covers of} to be the category of finite \'etale morphisms  $(U_1,\mf{X}_1,j_1)\to (U,\mf{X},j)$, with morphisms being morphisms of gluing triples over $(U,\mf{X},j)$.
\end{defin}

The following is an immediate corollary of Theorem \ref{thm:main} together with Proposition \ref{prop:triples-morphisms-properties}.

\begin{prop}\label{prop:FEt-gluing} 
    Let $X$ be an algebraic space separated and locally of finite type over $S$. Then,
    \begin{equation*}
        \mct\colon \FEt(X)\to \FEt(\mct(X)),
    \end{equation*}
    is an equivalence of categories.
\end{prop}

Proposition \ref{prop:FEt-gluing} implies a mixed characteristic analogue of the Riemann existence theorem in rigid geometry (see \cite{LutkebohmertRET}). Throughout the rest of this section, we fix the following data:
    \begin{itemize}
    \item $K$ a complete discretely valued field of characteristic $p\geqslant 0$ 
    \item $\mcO$ is the valuation ring of $K$,
    \item $k$ is the residue field of $K$,
    \item $\pi$ a uniformizer of $K$,
    \item $X$ is a scheme separated and locally of finite type over $\mcO$.
    \end{itemize}
\noindent We associate to $X$ an adic space over $\Spa(\mcO)$ built from the associated triple $\mct(X)$ in a natural way, and which \emph{differs} in the special fiber from the normal `analytification' of $X$
\begin{equation*}
    X^\mr{ad}= X\times_{\Spec(\mcO)}\Spa(\mcO),
\end{equation*}
as in \cite[Proposition 3.8]{HuberEC}.

\begin{defin} 
    We define the adic space $X^{\backslash\mr{ad}}$ over $\Spa(\mcO)$ by the following construction:
    \begin{equation*}
        X^{\backslash\mr{ad}}= (X^\circ)^\an\sqcup_{j_X,\wh{X}^\rig,\text{incl.}}\wh{X}^\mr{ad},
    \end{equation*}
    which is well-defined as $j_X$ and the natural inclusion $\wh{X}^\rig\to \wh{X}^\mr{ad}$ are open embeddings. 
\end{defin}

\begin{remark} 
Such spaces are closely related to recent work in the theory of Shimura varieties, e.g., see \cite{PappasRapoportI} as well as \cite{IKY1} and \cite{IKY2}. Indeed, if one considers the $v$-sheaf $(X^{\backslash\mr{ad}})^\lozenge$ associated to $X^{\backslash\mr{ad}}$ as in \cite[\S18.1]{ScholzeBerkeley} this agrees with the $v$-sheaf $X^{\backslash\lozenge}$ from \cite[Definition 2.1.9]{PappasRapoportI}. 
\end{remark}


\begin{example} 
If $X=\mathbb{A}^1_{\mcO}$, then the pullback of $X^{\backslash\mr{ad}}$ and $X^{\mr{ad}}$ to $\Spa(k)$ is $\Spa(k[x],k[x])$ and $\Spa(k[x],\mathbb{Z})$, respectively.
\end{example}

Let $C$ be an algebraically closed non-archimedean extension of $K$, and let $x\colon \Spec(\mcO_C)\to X$ be morphism over $\mcO_K$. As $\mcO_C$ has no finite \'etale covers, we have natural fiber functors 
\begin{equation*}
    F_x\colon \FEt(X)\to \cat{Set},\qquad F_{x^{\backslash\mr{ad}}}\colon \FEt(X^{\backslash\mr{ad}})\to\cat{Set}.
\end{equation*}
These give rise to the fundamental groups
\begin{equation*}
    \pi_1^\et(X,x)=\mr{Aut}(F_x),\qquad \pi_1^\et(X^{\backslash\mr{ad}},x^{\backslash\mr{ad}})=\mr{Aut}(F_{x^{\backslash\mr{ad}}}),
\end{equation*}
respectively. The notation $(-)^{\prime p}$ means we restrict to the subcategory of finite \'etale covers with prime-to-$p$ order or the associated quotients of their fundamental groups.



\begin{prop}\label{prop:integral-RET} 
    The morphism $(-)^{\backslash\mr{ad}}\colon \FEt(X)^{\prime p}\to \FEt(X^{\backslash\mr{ad}})^{\prime p}$ is an equivalence. Thus, the induced morphism
    \begin{equation*}
        \pi_1(X,x)^{\prime p}\to \pi_1^\et(X^{\backslash\mr{ad}},x^{\backslash\mr{ad}})^{\prime p}
    \end{equation*}
    is an isomorphism.
\end{prop}
\begin{proof} We claim that there is a commutative diagram
\begin{equation*}
    \begin{tikzcd}[sep=2.25em]
	{\FEt(X)^{\prime p}} & {\FEt(X^{\backslash\mr{ad}})^{\prime p}} & {\FEt(X_K^\mr{an},\wh{X}^\mr{ad},j_X)^{\prime p}} \\
	&& {\FEt(X_K,\wh{X},j_X)^{\prime p},}
	\arrow[from=1-1, to=1-2]
	\arrow["\sim"', bend right=15, from=1-1, to=2-3]
	\arrow["\sim", from=1-2, to=1-3]
	\arrow["\wr"', from=2-3, to=1-3]
\end{tikzcd}
\end{equation*}
with the marked arrows being equivalences of categories. This implies the desired claim. The curved arrow is just the equivalence from Proposition \ref{prop:FEt-gluing}, and so it suffices to explain the horizontal and vertical marked arrows.

Here $\FEt(X_K^\mr{an},\wh{X}^\mr{ad},j_X)^{\prime p}$ means the category of triples $(\mathsf{X},\mathsf{Y},j)$ consisting of 
\begin{itemize} 
\item a prime-to-$p$ finite \'etale cover $\mathsf{X}\to X_K^\mr{an}$, 
\item a prime-to-$p$ finite \'etale cover $\mathsf{Y}\to \wh{X}^\mr{ad}$, 
\item and $j$ an isomorphism of their pullbacks to $\wh{X}^\mr{rig}$.
\end{itemize}
The horizontal marked arrow is then the obvious equivalence obtained from the fact that $X^{\backslash\mr{ad}}$ is the gluing of $X_K^\mr{an}$ and $\wh{X}^\mr{ad}$ along the open subset $\wh{X}^\mr{rig}$. 

The vertical map associates to $(U_1,\mf{X}_1,j_1)$ the triple $(U_1^\mr{an},\mf{X}_1^\mr{ad},j_1)$. That this is an equivalence follows from applying \cite[Theorem 4.1]{LutkebohmertRET} and the fact that the functor
\begin{equation*} (-)^\mr{ad}\colon \FEt(\wh{X})\to \FEt(\wh{X}^\mr{ad})
\end{equation*}
is an equivalence. For this latter claim we may reduce to the case when $X=\Spf(A)$. But, 
\begin{equation*} 
    \FEt(\Spf(A))\simeq \FEt(\Spec(A))\simeq \FEt(\Spa(A)),
\end{equation*}
explicitly sending $\Spec(B)\to \Spec(A)$ to $\Spf(B)\to\Spf(A)$ and $\Spa(B)\to\Spa(A)$ ($B$ receiving the $\varpi$-adic topology), respectively. The latter being an equivalence follows as in \cite[Example 1.6.6 ii)]{HuberEC}, and the former by combining \cite[Chapter I, Proposition 4.2.1]{FujiwaraKato} and \cite[Lemma A.13]{ALYSpecialization}.
\end{proof}

As a corollary of this we see that one can detect the algebraizability of a finite \'etale cover on the generic fiber, similar to Proposition \ref{prop:A-E-equiv}. Namely, fix a finite \'etale cover $\mf{Y}\to \wh{X}$. Set 
\begin{itemize} 
\item $\mc{A}(\mf{Y}\to \wh{X})$ to be the category of \emph{algebraizations} of $\mf{Y}$, i.e., pairs $(Y,\iota)$ where $Y$ is an algebraic space finite and \'etale over $S$ and $\iota\colon \mf{Y}\isomto \wh{Y}$ an isomorphism over $\wh{S}$,
\item $\mc{E}(\mf{Y}^\mr{rig}\to \wh{X}^\mr{rig})$ to be the category of pairs $(U,j)$ where $U$ is a finite \'etale cover of $X_K$ and $j$ is an \emph{isomorphism} of $\mf{Y}^\mr{rig}$ onto $U^\mr{an}|_{\mf{X}^\mr{rig}}\subseteq U^\mr{an}$ over $\mf{X}^\mr{rig}$.
\end{itemize}
The following proposition is an immediate corollary of Proposition \ref{prop:integral-RET}.

\begin{prop}\label{prop:FEt-algebraizability} The functor
    \begin{equation*}
        \mc{A}(\mf{Y}\to\mf{X})\to \mc{E}(\mf{Y}^\mr{rig}\to\mf{X}^\mr{rig}),\quad (Y,\iota)\mapsto (Y^\circ,j_Y\circ \iota^\mr{rig})
    \end{equation*}
    is an equivalence with quasi-inverse given by sending $(U,j)$ to $\mcg(U,\mf{Y},j)$. 
\end{prop}

\begin{cor}\label{cor:algebraizability} A prime-to-$p$ finite \'etale cover $\mf{Y}\to \wh{X}$ is algebraizable if and only if the finite \'etale cover $\mf{Y}^\mr{rig}\to \wh{X}^\mr{rig}$ extends to a finite \'etale cover of $X_K^\mr{an}$.
\end{cor}

\begin{remark}\label{rem:Grothendieck-specialization} 
Part of Grothendieck's construction of a specialization morphism for the \'etale fundamental group is the following result. If $X$ is a proper scheme over $\mcO$ then every finite \'etale cover $Y_k\to X_k$ deforms uniquely to a finite \'etale cover $Y\to X$. It is instructive to think about this result in two steps:
\begin{enumerate}[leftmargin=2cm, itemsep=3pt]
    \item[\textbf{Step 1:}] a finite \'etale cover $Y_k\to X_k$ deforms uniquely to a finite \'etale cover  $\mf{Y}\to \wh{X}$,
    \item[\textbf{Step 2:}] a finite \'etale cover $\mf{Y}\to\wh{X}$ is uniquely algebraizable to a finite \'etale cover $Y\to X$.
\end{enumerate}
It is \textbf{Step 2} that requires $X$ to be proper over $\mcO$ as it utilizes Grothendieck's existence theorem, and \textbf{Step 1} is true for any $X$. As $\wh{X}^{\mr{rig}}=(X^\circ)^\mr{an}$ when $X$ is proper over $\mcO$, one may view Proposition \ref{prop:FEt-algebraizability} and Corollary \ref{cor:algebraizability} as precisely qualifying the failure for \textbf{Step 2} to hold in the non-proper case, where generally $\wh{X}^\mr{rig}\subsetneq (X^\circ)^\mr{an}$, and providing the precise extra information required to mend it.
\end{remark}

\section{Coherent sheaves and gluing triples}
\label{s:coh-gluing}

In this section we show an equivalence of categories between coherent sheaves on an algebraic space $X$ and on its associated triple $\mct(X)$,  which is important in our proof of Theorem \ref{thm:main}. 

\subsection{Gluing for modules over an affine triple}
\label{ss:coh-gluing-affine} We first recall the gluing of results of Artin (see \cite[Theorem 2.6]{ArtinII}) in the Noetherian case and Beauville--Laszlo (see \cite{BeauvilleLaszlo}) in the principal case, which can be uniformly extended to the type (V) situation thanks to the methods in \stacks[Lemma]{0FGM}. This accounts roughly for our desired gluing of coherent sheaves in the affine case.

Throughout we fix an affine scheme $X=\Spec(A)$ of finite type over $S$ and set
\begin{equation*}
    I=\mc{I}(\Spec(A)),\qquad X^\circ=\Spec(A)\mysetminus V(I),\qquad \wh{X}^\circ= \Spec(\wh{A})\mysetminus V(I\wh{A}),
\end{equation*}
the first an ideal of definition of $A$ and the latter two schemes over $S^\circ$. 

The following lemma summarizes the ring-theoretic properties of $A$.

\begin{lemma}\label{lem:completion-properties} 
    If $(\Spec(A),V(I))$ is of Type (N)/(V), then the following properties hold.
    \begin{enumerate}
        \item The pair $(A,I)$ is universally pseudo-adhesive, and the ring $A$ is coherent,\footnote{Recall (e.g., see \cite[Chapter 0, Definitions 8.5.1 and 8.5.4]{FujiwaraKato}), that a ring/ideal pair $(A,I)$ is called \emph{universally pseudo-adhesive} if for any finite type $A$-algebra $R$ the scheme $\Spec(R)\mysetminus V(IR)$ is Noetherian, and for any finitely generated $R$-module $M$ the module $M[I^\infty]$ is bounded, i.e., there exists some $n\geqslant 0$ such that $I^n M[I^\infty]=0$. Moreover, $A$ is called \emph{coherent} if every finitely generated ideal is finitely presented, in which case the categories of coherent $A$-modules and finitely presented $A$-modules coincide.}
        \item the map $X^\circ\sqcup \Spec(\wh{A})\to X$ is faithfully flat,
        \item for any finite $A$-module $M$ the natural map $M\otimes_A \wh{A}\to \wh{M}$ is an isomorphism,
        \item for any finite $A$-module $M$, the natural map $M[I^\infty]\to \wh{M}[I^\infty]$ is a bijection.
    \end{enumerate}
\end{lemma}

\begin{proof} 
For the pseudo-adhesiveness claim in (1) see \cite[Chapter 0, Example 8.5.13]{FujiwaraKato} for type (N), and \cite[Theorem 2.16]{ZavyalovSheafiness} and/or \cite[Chapter 0, Theorem 9.2.1]{FujiwaraKato} for type (V), respectively. The coherence claim in (1) is clear in Type (N), and follows from \cite[Chapter 0, Corollary 9.2.9]{FujiwaraKato} in case of Type (V). Given (1), the claim in (3) follows from \cite[Chapter 0, Proposition 8.2.18]{FujiwaraKato}, and in fact this reference also implies (2). Indeed, evidently $X^\circ\sqcup \Spec(\wh{A})\to X$ is surjective and and $X^\circ\to X$ is flat, so it suffices to explain why $A\to\wh{A}$ is flat, but this follows from loc.\@ cit.


Finally, we prove (4). Observe that we have a short exact sequence
\begin{equation*}
    0\to M[I^\infty]\to M\to Q\to 0,
\end{equation*}
where $Q$ is a finitely generated and $I$-torsion free $A$-module. By work of Raynaud--Gruson (see \cite[\S7.4, Theorem 4]{BoschLectures}) $Q$ is then finitely presented, and so $M[I^\infty]$ is actually finitely generated (see \stacks[Lemma]{0519}). As $A\to\wh{A}$ is flat and $M\otimes_A \wh{A}\simeq \wh{M}$, we obtain an exact sequence
\begin{equation*}
    0\to M[I^\infty]\otimes_A\wh{A}\to \wh{M}\to Q\otimes_A\wh{A}\to 0.
\end{equation*}
As $M[I^\infty]$ and $Q$ are finitely generated $A$-modules, we have isomorphisms
\begin{equation}\label{eq:completion-and-tensor}
    M[I^\infty]\otimes_A\wh{A}\simeq {A[I^\infty]}^\wedge,\quad Q\otimes_A\wh{A}\simeq \wh{Q},
\end{equation}
where all completions are $I$-adic. But $M[I^\infty]$ is $I$-torsion and so $M[I^\infty]^\wedge=M[I^\infty]$. So, all in all, we have the exact sequence
\begin{equation*}
      0\to M[I^\infty]\to \wh{M}\to \wh{Q}\to 0.
\end{equation*}
So, it suffices to observe that $\wh{Q}$ is $I$-torsion free. By \eqref{eq:completion-and-tensor} this follows as $A\to \wh{A}$ is flat.
\end{proof}

In this affine setting, the category $\cat{Coh}(\Spec(A))$ of coherent sheaves on $\Spec(A)$ is replaced by the category $\cat{Mod}^\mr{fp}(A)$ of finitely-presented $A$-modules. The analogue for the gluing triple $\mct(A)$ is given by the following.

\begin{defin} 
    Denote by $\cat{Mod}^\mr{fp}(\mct(A))$ the category of triples $(\mcF,N,\iota)$ where 
    \begin{enumerate} 
        \item $\mcF$ is a coherent sheaf on $X^\circ$,
        \item $N$ is a finitely presented $\wh{A}$-module,
        \item $\iota\colon \mcF|_{\wh{X}^\circ}\isomto \wt{N}|_{\wh{X}^\circ}$ is an isomorphism of coherent sheaves on $\wh{X}^\circ$,
    \end{enumerate}
    with a morphism $(\mcF_1,N_1,\iota_1)\to (\mcF,N,\iota)$ a pair $\alpha\colon \mcF_1\to \mcF$ and $\beta\colon N_1\to N$ of morphisms such that $\iota\circ \alpha|_{\wh{X}^\circ}=\wt{\beta}|_{\wh{X}^\circ}\circ\iota_1$.
\end{defin}

We will often abuse notation concerning a triple $(\mcF,N,\iota)$ and let $\wh{\mcF}$ denote the common (via the identification $\iota$) coherent sheaf on $\wh{X}^\circ$.

\begin{prop}\label{prop:coh-gluing-affine-case} 
    The functor 
    \begin{equation*}
        \mct\colon \cat{Mod}^{\mr{fp}}(A)\to \cat{Mod}^\mr{fp}(\mct(A)),\qquad M\mapsto (\wt{M}|_{X^\circ},\wh{M},\mr{can.}),
    \end{equation*}
    is an equivalence of $A$-linear abelian $\otimes$-categories, with quasi-inverse given by the \emph{gluing functor}
    \begin{equation*}
        \mcg\colon \cat{Mod}^\mr{fp}(\mct(A))\to \cat{Mod}^\mr{fp}(\mct(A)),\quad (\mcF,N,\iota)\mapsto \mcF(X^\circ)\times_{\wh{\mcF}(\wh{X}^\circ)}N.
    \end{equation*}
\end{prop}

\begin{proof}
In type (N), this has been established by Artin in \cite[Theorem 2.6]{ArtinII}. In order to cover type (V) as well, we aim to employ the enhanced version of the Beauville--Laszlo theorem in \stacks[Theorem]{0BP2}. Type (N) already covered, we may assume that the ideal of definition $I$ is principal, generated by an element $\pi$. Assertions (3) and (4) of Lemma~\ref{lem:completion-properties} imply that every finitely generated $A$-module is glueable (as defined in \stacks[Section]{0BNI}). It remains to show that an $A$-module $M$ is finitely presented if both $\widehat{M}$ and $M[\nicefrac{1}{\pi}]$ are. Since the map $A\to \widehat{A}\times A[\nicefrac{1}{\pi}]$ is faithfully flat by Lemma~\ref{lem:completion-properties}, we conclude by \stacks[Lemma]{05B0}.
\end{proof}

Utilizing the equivalence of coherent modules on $\wh{X}^\mr{rig}$ and $\wh{X}^\circ$ as in \cite[Chapter II, Proposition 6.6.5]{FujiwaraKato}, the following is an immediate corollary of Proposition~\ref{prop:coh-gluing-affine-case}. 


\begin{cor}\label{cor:A-reconstruction} 
The natural map
    \begin{equation}\label{eq:function-reconstruction}
        \mcO(X)\to \mcO(X^\circ)\times_{\mcO(\wh{X}^\mr{rig})}\mcO(\wh{X}),
    \end{equation}
    is an isomorphism of rings.
\end{cor}

\subsection{Coherent sheaves and triples} 

We now patch together the affine-local results of \S\ref{ss:coh-gluing-affine}. This gives a globalization of the previously mentioned results of Artin and Beauville--Laszlo, which can only happen involving analytic geometry, and generalizes similar ideas in \cite{BBT}.

To set notation, we first recall the definition of coherent sheaves on algebraic spaces. In the following, $(-)_\Et$ and $(-)_\mr{top}$ denotes the big sites with \'etale and topological open covers, respectively, and $\mcO_\ms{X}$ is as in \eqref{eq:structure-sheaf-def}.

\begin{defin}\label{def:coh-on-alg-sp} Let $\ms{X}$ be a (formal) algebraic space or analytic adic algebraic space of type (N)/(V). The category $\cat{Coh}(\ms{X})$ of \emph{coherent sheaves on $\ms{X}$} is the full subcategory of $\cat{Mod}(\ms{X}_\Et,\mcO_\ms{X})$ consisting of finitely presented $\mcO_\ms{X}$-modules (in the sense of \stacks[Definition]{01BN}).
\end{defin}

Defining $\cat{Coh}^\mr{top}(\ms{X})$ using $(-)_\mr{top}$ instead of $(-)_\Et$, we have the following descent result

\begin{prop}\label{prop:coh-descent} Let $\ms{U}\to\ms{X}$ be an \'etale surjection with $\ms{U}$ representable. Then, the functor
\begin{equation*}
    \cat{Coh}(\ms{X})\to 2\text{-}\mr{Eq}\bigg(\cat{Coh}^\mr{top}(\ms{U})\mathrel{\substack{\textstyle\rightarrow\\[-0.6ex]
                      \textstyle\rightarrow}}  \cat{Coh}^\mr{top}(\ms{U}\times_\ms{X}\ms{U})\mathrel{\substack{\textstyle\rightarrow\\[-0.6ex]
                      \textstyle\rightarrow \\[-0.6ex]
                      \textstyle\rightarrow}}  \cat{Coh}^\mr{top}(\ms{U}\times_\ms{X}\ms{U}\times_\ms{X}\ms{U})\bigg),
\end{equation*}
is a natural equivalence. Moreover, the global sections functors
\begin{equation*}
 \cat{Coh}(\Spf(A))\to\cat{Mod}^\mr{fp}(A), \qquad \cat{Coh}(\Spa(A,A^+))\to\cat{Mod}^\mr{fp}(A)
\end{equation*}
are natural equivalences of $A$-linear $\otimes$-categories.
\end{prop}
\begin{proof} The first claim follows from \cite[Chapter I, Proposition 6.1.11]{FujiwaraKato} and \cite[Theorem 2.1]{BoschGortz} in the case of (formal) schemes and analytic adic space of type (N)/(V), respectively. The second pair of claimed equivalences follows, given the coherence of $A$ from Lemma \ref{lem:completion-properties}, follows from \cite[Chapter I, Theorem 3.2.8]{FujiwaraKato} and \cite[Theorem 1.4.18]{KedlayaAWS} for (formal) schemes and analytic adic spaces of type (N)/(V), respectively.
\end{proof}

Reducing to the affine situation and applying Proposition \ref{prop:coh-descent}, one deduces the following.

\begin{lemma} 
    For $\ms{X}$ as in Definition \ref{def:coh-on-alg-sp}, $\cat{Coh}(\ms{X})$ is an $\mcO_{\ms{X}}$-linear abelian $\otimes$-category.
\end{lemma}

Observe that we have the following two natural functors.
\begin{enumerate}
    \item For a formal algebraic space $\mf{X}$ locally of finite type over $\wh{S}$ we have a functor
\begin{equation*}
    (-)^\mr{rig}\colon \cat{Coh}(\mf{X})\to \cat{Coh}(\mf{X}^\mr{rig}),\quad \mf{F}\mapsto \mf{F}^\mr{rig}=\mr{sp}_{\mf{X}}^\ast(\mf{F})\otimes_{\mcO_{\mf{X}^\mr{rig}}^+}\mcO_{\mf{X}^\mr{rig}}.
\end{equation*}
\item If $U$ is an algebraic space locally of finite type over $S^\circ$ there is a map of ringed topoi 
\begin{equation*}
    (-)^\mr{an}\colon (U^\mr{an}_\Et,\mcO_{U^\mr{an}})\to (U_\Et,\mcO_{U}),
\end{equation*}
which induces a natural functor
\begin{equation*}
    (-)^\mr{an}\colon \cat{Coh}(U)\to\cat{Coh}(U^\mr{an}).
\end{equation*}
\end{enumerate}

We are now prepared to define coherent sheaves on gluing triples over $(S,S_0)$.

\begin{defin}\label{def:coh-sheaves-on-triples} 
    Let $(U,\mf{X},j)$ be a gluing triple over $(S,S_0)$. Then, we define the category $\cat{Coh}(U,\mf{X},j)$ of \emph{coherent sheaves} to sit in the following $2$-cartesian diagram of categories: 
    \begin{equation*}
    \begin{tikzcd}
    	{\cat{Coh}(U,\mf{X},j)} & {\cat{Coh}(U)} \\
    	\\
    	{\cat{Coh}(\mf{X})} & {\cat{Coh}(\mf{X}^\mr{rig}),}
    	\arrow[from=1-1, to=1-2]
    	\arrow[from=1-1, to=3-1]
    	\arrow["\Bigg\lrcorner"{anchor=center, pos=0.125}, draw=none, from=1-1, to=3-2]
    	\arrow["{j^\ast\circ (-)^\mr{an}}", from=1-2, to=3-2]
    	\arrow["{(-)^\mr{rig}}"', from=3-1, to=3-2]
    \end{tikzcd}
    \end{equation*}
    i.e., the category of triples $(\mcF,\mf{F},\iota)$ where $\mcF$ is a coherent $\mcO_{U}$-module, $\mf{F}$ is a coherent $\mcO_{\mf{X}}$-module, and $\iota\colon \mf{F}^\mr{rig}\isomto j^\ast(\mcF^\mr{an})$ is an isomorphism of coherent $\mcO_{\mf{X}^\mr{rig}}$-modules.
\end{defin}

If $X$ is an algebraic space locally of finite type over $S$ there is a natural functor
\begin{equation*}
    \mct\colon \cat{Coh}(X)\to \cat{Coh}(\mct(X)),\quad F\mapsto \mct(F)=(F^\circ,\wh{F},\iota_F).
\end{equation*}
Here we denote by 
\begin{enumerate}[(i),itemsep=3pt] 
    \item $(-)^\circ$ the pullback functor along the morphism of schemes $X^\circ\to X$,
    \item $\wh{(-)}$ the pullback functor along the morphism of ringed topoi $\wh{(-)}\colon (\wh{X}_\Et,\mcO_{\wh{X}})\to (X_\Et,\mcO_X)$, 
    \item and $\iota_F\colon \wh{F}^\mr{rig}\isomto j_X^\ast((F^\circ)^\mr{an})$ is obtained as both are naturally the pullback of $F$ along the morphism of ringed topoi $(\wh{X}^\mr{rig}_\Et,\mcO_{\wh{X}^\mr{rig}})\to (X_\Et,\mcO_X)$. 
\end{enumerate}
The functor $\mct$ is clearly $2$-functorial in $X$.

\begin{prop}\label{prop:coh-gluing} 
    Let $X$ be a algebraic space separated and locally of finite type over $S$. Then,
    \begin{equation*}
        \mct\colon \cat{Coh}(X)\to \cat{Coh}(\mct(X)),
    \end{equation*}
    is an $\mcO_X$-linear $\otimes$-equivalence of abelian categories.
\end{prop}

\begin{proof}
Observe that coherent sheaves on gluing triples satisfy \'etale descent as quickly follows from applying Proposition \ref{prop:coh-descent} component-wise. More precisely, the following holds.

\begin{lemma}\label{lem:etale-descent-for-triples} 
    Suppose that $(U_0,\mf{X}_0,j_0)\to (U,\mf{X},j)$ is an \'etale surjection of gluing triples over $(S,S_0)$. Write $(U_1,\mf{X}_1,j_1)$ and $(U_2,\mf{X}_2,j_2)$ for the self fiber product and self triple fiber product of $(U_0,\mf{X}_0,j_0)$ over $(U,\mf{X},j)$, respectively. Then, the natural functor
    \begin{equation*}
        \cat{Coh}(U,\mf{X},j)\to 2\text{-}\mr{Eq}\left(\cat{Coh}(U_0,\mf{X}_0,j_0)\mathrel{\substack{\textstyle\rightarrow\\[-0.6ex]
                          \textstyle\rightarrow}}  \cat{Coh}(U_1,\mf{X}_1,j_1)\mathrel{\substack{\textstyle\rightarrow\\[-0.6ex]
                          \textstyle\rightarrow \\[-0.6ex]
                          \textstyle\rightarrow}}  \cat{Coh}(U_2,\mf{X}_2,j_2)\right),
    \end{equation*}
    is an equivalence of categories.
\end{lemma}

Applying this and the first part of Proposition \ref{prop:coh-descent} to an \'etale cover $\{\Spec(A_i)\to X\}$ quickly reduces us to the case when $X=\Spec(A)$. Write $I=\mc{I}(\Spec(A))$, and recall from \S\ref{ss:coh-gluing-affine} that $\wh{X}^\circ$ is defined to be $\Spec(\wh{A})\mysetminus V(I\wh{A})$. There is a natural morphism of ringed spaces $(\wh{X}^\mr{rig},\mcO_{\wh{X}^\mr{rig}})\to (\wh{X}^\circ,\mcO_{\wh{X}^\circ})$ which induces a pullback $\mcO_X$-linear $\otimes$-equivalence of abelian categories of coherent sheaves (see \cite[Chapter II, Proposition 6.6.5]{FujiwaraKato}). Combining this and the second part of Proposition \ref{prop:coh-descent}, we see that the claim for $X=\Spec(A)$ is reduced precisely to Proposition~\ref{prop:coh-gluing-affine-case}, from where the claim follows.
\end{proof}

\begin{cor}\label{cor:vb-gluing} 
    Let $X$ be an algebraic space separated and locally of finite type over $S$. Then,
    \begin{equation*}
        \mct\colon \cat{Vect}(X)\to \cat{Vect}(\mct(X)),
    \end{equation*}
    is a Quillen-exact $\mcO_X$-linear $\otimes$-equivalence of categories with Quillen-exact quasi-inverse.
\end{cor}

\begin{proof} 
Given Proposition \ref{prop:coh-gluing}, it suffices to show that a coherent $\mcO_X$-module $F$ is a vector bundle if and only if $\mct(F)$ is. But, we may quickly reduce to the case when $X=\Spec(A)$ is affine. Let us write $M=F(X)$. Note that $\wh{F}$ is a vector bundle if and only if $\wh{F}(\wh{X})=\wh{M}$ is a finite projective $\wh{A}$-module (e.g., see the proof of \cite[Proposition A.13]{IKY1}). As $\wh{M}=M\otimes_A \wh{A}$ by Lemma \ref{lem:completion-properties}, we are thus reduced to the following claim: $M$ is a finite projective $A$-module if and only if $\wh{M}|_{X^\circ}$ and $M\otimes_A \wh{A}$ is a vector bundle on $X^\circ$ and a finite projective $\wh{A}$-module, respectively. But, given Lemma \ref{lem:completion-properties} this follows from \stacks[Theorem]{05A9}.
\end{proof}

\section{Proofs}
\label{s:proof-of-gluing-theorem}

We now turn to the proof of Theorem \ref{thm:main}. For ease of reading, we have broken this proof down into steps, comprising the first three subsections. In the remaining two subsections we provide the only remaining missing proofs, that of Propositions \ref{prop:triples-morphisms-properties} and \ref{prop:affine-gluing-base-case}.

\subsection{Explicit reconstruction in the schematic case} 

We first observe a more down-to-earth case of Theorem \ref{thm:main}. Namely, we wish to show that if $X$ is a scheme locally of finite type over $S$ then one may explicitly reconstruct the ringed space $(X,\mcO_X)$ from the triple $\mct(X)$.

\begin{defin} 
    Let $(U,\mf{X},j)$ be gluing triple over $(S,S_0)$. The \emph{underlying topological space}, denoted $|(U,\mf{X},j)|$, has underlying set $|U|\sqcup |\mf{X}|$, where a subset $F = F_U\sqcup F_{\mf{X}}$ is closed if and only if $F_U = F\cap |U|$ is closed in $|U|$, $F_{\mf{X}} = F\cap|\mf{X}|$ is closed in $|\mf{X}|$, and \[ F^{\rm an}_U\cap \mf{X}^{\rm rig} \subseteq {\rm sp}_{\mf{X}}^{-1}(F_{\mf{X}}). \]
\end{defin}

In other words, $F$ is closed if and only if it is of the form $C\cup \mathrm{sp}_\mf{X}(C^\an\cap \mf{X}^\mr{rig})\cup D$, where $C$ is a closed subset of $U$ and $D$ a closed subset of $\mf{X}$. It is clear that the construction of the underlying topological space determines a functor
\begin{equation*}
    |\cdot|\colon \cat{Trip}_{(S,S_0)}\to \cat{Top},
\end{equation*}
and that there is a natural map of topological spaces $\alpha_X\colon |\mct(X)|\to |X|$.

\begin{prop}\label{prop:triple-scheme-same-space-different-day} 
    Let $X$ be an algebraic space locally of finite type over $S$. Then, the map 
    \begin{equation*} 
        \alpha_X\colon |\mct(X)|\to |X|
    \end{equation*}
    is a homeomorphism.
\end{prop}

\begin{proof} 
As the map $\alpha_X$ is continuous and bijective, it suffices to show this map is closed. 

First assume that $X$ is a scheme. It suffices to prove that if $C$ is a closed subset of $X^\circ$, then $C\sqcup \mathrm{sp}_{\wh{X}}(C^\an\cap \wh{X}^\mr{rig})$ is a closed subset of $X$. As admissible blowups are closed maps and the specialization map is functorial, this quickly reduces to the case when $X=\Spec(A)$ where $\mc{I}(X)=I$ is principal.

Suppose that $C=V(J)$ for $J\subseteq \mcO(X^\circ)$ an ideal. We claim that $C\sqcup \mathrm{sp}_{\wh{X}}(C^\an\cap \wh{X}^\mr{rig})=V(J^c)$ where $J^c$ is the pullback of $J$ along the map $A\to \mcO(X^\circ)$. This reduces to showing that $\mathrm{sp}_{\wh{X}}(V(J^c)^\mathrm{an}\cap\wh{X}^\mr{rig})=V(I,J^c)$. Set $B=A/J^c$, then this is equivalent to the claim that $\mathrm{sp}_{\wh{X}}(\Spf(\wh{B})^\mr{rig})=\Spec(B/IB)$. By \cite[Chapter II, Proposition 3.1.5]{FujiwaraKato} it suffices to show that $\wh{B}$ is $I$-torsionfree. By Lemma \ref{lem:completion-properties}, $\Spec(\wh{B})\to\Spec(B)$ is flat, and so it suffices to show that $B$ is $I$-torsionfree. But, $B=A/J^c$ injects into $\mcO(X^\circ)/J$ which is $I$-torsionfree.

To show closedness in general, let $X_1\to X$ be an \'etale surjection with $X_1$ a scheme, so that $|X_1|\to |X|$ is a quotient map. By functoriality of the specialization map, the pullback of $C\sqcup \mathrm{sp}_{\wh{X}}(C^\an\cap \wh{X}^\mr{rig})$ under $|X_1|\to |X|$ is $C_1\sqcup \mathrm{sp}_{\wh{X}_1}(C_1^\an\cap \wh{X}_1^\mr{rig})$, where $C_1$ is the preimage of $C$ under $X_1^\circ\to X^\circ$. This set is closed by previous argument, and thus so is $C\sqcup \mathrm{sp}_{\wh{X}}(C^\an\cap \wh{X}^\mr{rig})$.
\end{proof}

Define the full subcategory $\mathscr{C}\subseteq \cat{Trip}_{(S,S_0)}$ to be the essential image under $\mct$ of representable objects, i.e., those $(U,\mf{X},j)$ isomorphic to $\mct(X)$ for some scheme $X$. Let us observe by Proposition \ref{prop:triple-scheme-same-space-different-day} one has a canonical identification $|(U,\mf{X},j)|\simeq |X|$. We now define a functor
\begin{equation*}
    \mcg\colon \mathscr{C}\to \cat{RS}_S,\quad \mcg(U,\mf{X},j)=(|(U,\mf{X},j)|,\mcO_{(U,\mf{X},j)}).
\end{equation*}
Here $\cat{RS}_S$ is the category of ringed spaces over $S$, and we define the sheaf $\mcO_{(U,\mf{X},j)}$ as follows: 
\begin{equation*}
    \mcO_{(U,\mf{X},j)}(|\mct(X_1)|)= \mcO_U(X_1)\times_{\mcO(\wh{X}_1^\mr{rig})}\mcO(\wh{X}_1),
\end{equation*}
for any open subset $X_1\subseteq X$, which is clearly a well-defined sheaf on $|(U,\mf{X},j)|\simeq |X|$ (as all objects involved only require the knowledge of the underlying space of $X_1$), independent of choices, and that $\mcg$ is a functor. 

\begin{prop}\label{prop:scheme-reconstruction} 
    For a scheme $X$ locally of finite type over $S$, there is a natural isomorphism of ringed spaces over $S$
    \begin{equation*}
        \alpha_X\colon (\mct(X),\mcO_{\mct(X)})\isomto (X,\mcO_X).
    \end{equation*}
\end{prop}

\begin{proof} 
By Proposition \ref{prop:triple-scheme-same-space-different-day}, $\alpha_X$ is a homeomorphism, so it suffices to show that the natural map $\mcO_X\to (\alpha_X)_\ast\mcO_{\mct(X)}$ is an isomorphism on affine open subsets. But, this is Corollary \ref{cor:A-reconstruction}.
\end{proof}

\begin{cor}\label{cor:fully-faithfulness-schemes} 
    The functor 
    \begin{equation*}
        \mct\colon \cat{Sch}^\mr{lft}_S\to \cat{Trip}_{(S,S_0)}
    \end{equation*}
    is fully faithful.
\end{cor}

\subsection{Proof of Theorem \ref{thm:main}: fully faithfulness} 

We now aim to show that for any algebraic spaces  $X$ and $Y$ separated and locally of finite type over $S$, the natural map
\begin{equation}\label{eq:ff}
    \Hom(X,Y)\to \Hom(\mct(X),\mct(Y)),
\end{equation}
is a bijection. We do this in three steps.

\medskip

\paragraph{Step 0: when $X$ and $Y$ are both schemes} In this setting we see that the claim is a special case of Corollary \ref{cor:fully-faithfulness-schemes}.

\medskip

\paragraph{Step 1: when $X$ is a scheme} For injectivity, let $Y'\to Y$ be an \'etale surjection from a scheme $Y'$ separated and locally of finite type over $S$. If $f_i\colon X\to Y$ ($i=1,2$) are morphisms with $\mct(f_1)=\mct(f_2)$ then $\mct(f_1')=\mct(f_2')$, where $f_i'\colon Y'\times_Y X\to Y'$ are the induced morphisms. But, as $Y'\times_Y X$ is a scheme, this implies by Corollary \ref{cor:fully-faithfulness-schemes} that $f_1'=f_2'$ and, as $Y'\times_X Y\to X$ is a~surjection, that $f_1=f_2$.

For surjectivity, let $f\colon \mct(X)\to\mct(Y)$ be a morphism. As $X$ and $Y$ are separated the graph morphisms $\Gamma_{f^\circ}\colon X^\circ\to X^\circ\times_{S^\circ}Y^\circ$ and $\Gamma_{\wh{f}}\colon \wh{X}\to \wh{X}\times_{\wh{S}}\wh{Y}$ are closed embeddings, and so define ideal sheaves $\mc{I}_{f^\circ}\subseteq \mcO_{X^\circ\times_{S^\circ}Y^\circ}$ and $\mc{I}_{\wh{f}}\subseteq \mcO_{\wh{X}\times_{\wh{S}}\wh{Y}}$. As $f^\circ|_{\wh{X}^\mr{rig}}=\wh{f}|_{\wh{X}^\mr{rig}}$ it is clear that $\mc{I}_{f^\circ}$ and $\mc{I}_{\wh{f}}$ induce the same ideal sheaf of $\wh{X}^\mr{rig}\times_{\wh{S}^\mr{rig}}\wh{Y}^\mr{rig}$, so define an object $\mc{I}_f$ in $\cat{Coh}(\mct(X\times_S Y))$ in the sense of Definition \ref{def:coh-sheaves-on-triples}. By Proposition \ref{prop:coh-gluing} this corresponds to a unique coherent sheaf $\mc{J}$ (which is an ideal sheaf as an equivalence preserves subobjects) on $X\times_S Y$ with $\mct(\mc{J})=\mc{I}_f$. 

Consider $V(\mc{J})\subseteq X\times_S Y$. The projection map furnishes a morphism $V(\mc{J})\to X$ which is separated (as each $V(\mc{J})$ and $X$ is) and quasi-finite (as this may be checked over $X^\circ$ and $\wh{X}$, where it is an isomorphism). Thus, $V(\mc{J})$ is a scheme by \stacks[Proposition]{03XX}. Observe then that there is a natural isomorphism $\iota\colon \mct(X)\to \mct(V(\mc{J}))$ and thus by Corollary \ref{cor:fully-faithfulness-schemes} there exists a unique isomorphism $i\colon X\to V(\mc{J})$ such that $\mct(i)=\iota$. Composing $i$ with the projection $X\times_S Y\to Y$ produces a morphism $g\colon X\to Y$ which, by construction, satisfies $\mct(g)=f$, as desired. 

\medskip

\paragraph{Step 2: the general case}

By \cite[Expos\'e I, Corollaire 4.1.1]{SGA4-1} we may write $X=\varinjlim X_i$ in the category of presheaves on $S$-schemes where each $X_i$ is a scheme over $S$. In particular, we have an identification 
\begin{equation*}
    \Hom(X,Y)\simeq \varprojlim_{X_i\to X}\Hom(X_i,Y).
\end{equation*}
But, using \textbf{Step 1} one may directly check that one also has an identification
\begin{equation*}
    \Hom(\mct(X),\mct(Y))\simeq \varprojlim_{X_i\to X}\Hom(\mct(X_i),\mct(Y)).
\end{equation*}
The morphism from \eqref{eq:ff} is then obtained by passing to the limit over the morphisms 
\begin{equation*}
    \Hom(X_i,Y)\to \Hom(\mct(X_i),\mct(Y)),
\end{equation*}
which are each bijections by \textbf{Step 1}, and thus we are done.

\subsection{\texorpdfstring{Proof of the gluing theorem: essential surjectivity when $S$ is a $G$-space}{Proof of the gluing theorem: essential surjectivity in the G-space case}}
\label{ss:ess-surj-excellent-case} 

Our proof of essential surjectivity uses results relying on Artin approximation which requires excellence-like hypotheses. So, we now restrict to the case when $S$ is a $G$-space (see Definition \ref{def:G-space}).

\medskip

\subsubsection*{Artin's contraction theorem} 

We begin by recalling the main theorem of \cite{ArtinII}, and the improvement in \stacks[Theorem]{0GIB}. Our exposition will follow that of \stacks[Section]{0GH7}.

\begin{defin}\label{def:formal-contraction} 
    A \emph{formal contraction problem} over $S$ is a quadruple of data $(X',T',\mf{X},\mf{g})$ where
    \begin{itemize}
        \item $X'$ is an algebraic space locally of finite over $S$,
        \item $T'$ is a closed subset of $|X'|$,
        \item $\mf{X}$ is an formal algebraic space locally of finite type over $\wh{S}$,
        \item $\mf{g}\colon \wh{X}'_{T'}\to \mf{X}$ is a proper and rig-\'etale (see \stacks[Definition]{0AQM}) morphism of formal schemes such that both $\mf{g}$ and its diagonal are rig-surjective (see \stacks[Definition]{0AQQ}). Here $\wh{X}'_{T'}$ denotes the formal completion of $X'$ along $T'$. 
    \end{itemize}
\end{defin}

We will only need the following very special case of a formal contraction problem.

\begin{example} 
Let $X'$ be an algebraic space locally of finite type over $S$, and let $\mf{g}\colon \wh{X}'\to \wh{X}$ be a formal modification of formal algebraic spaces over $\wh{S}$. Then, $(X',|\wh{X}'|,\wh{X},\mf{g})$ is a formal contraction problem. Indeed, it suffices to show that $\mf{g}$ is rig-\'etale and that $\mf{g}$ and its diagonal are rig-surjective. But, by assumption, $\mf{g}^\mr{rig}$ is an isomorphism, and thus so is its diagonal. It is easy to see this implies that $\mf{g}$ and its diagonal are rig-surjective (cf.\@ \cite[Chapter II, Lemma 3.3.7]{FujiwaraKato}). To prove that $\mf{g}$ is rig-\'etale, we are reduced to the case of admissible blowups by combining \cite[Chapter II, Corollary 2.1.5]{FujiwaraKato} and part (4) of \stacks[Lemma]{0GCZ}, but this case is simple.
\end{example}

\begin{remark}
    If $S$ is Jacobson, then the converse holds: the conditions on the map $\mf{g}$ in Definition~\ref{def:formal-contraction} are equivalent to saying that the induced map of rigid spaces $\mf{g}^\rig$ is an isomorphism.
\end{remark}

\begin{defin}\label{def:Artin-contraction}
    Let $(X',T',\mf{X},\mf{g})$ be a formal contraction problem over $S$. Then, a \emph{potential solution} to this formal contraction problem is a quadruple $(X,T,g,\iota)$ where
    \begin{itemize}
        \item $X$ is an algebraic space locally of finite type over $S$,
        \item $g\colon X'\to X$ a proper morphism of algebraic spaces over $S$,
        \item $T\subseteq |X|$ a closed subset,
        \item $\iota\colon \wh{X}_{T}\isomto \mf{X}$ is an isomorphism of formal algebraic spaces over $\wh{S}$.
    \end{itemize}
    We say a potential solution $(X,T,g,\iota)$ is a \emph{solution} if 
    \begin{itemize}
        \item $T'=g^{-1}(T)$,
        \item $g\colon X'\mysetminus T'\to X\mysetminus T$ is an isomorphism,
        \item the completion of $g$ recovers, via $\iota$, the map $\mf{g}$.
    \end{itemize}
\end{defin}

In \cite{ArtinII} it is shown that every formal contraction problem has a solution when $S$ is excellent. This result was strengthened in \stacks[Section]{0GH7} to the case when $S$ is a $G$-space.

\begin{thm}[{Artin's contraction theorem, see \cite[Theorem 3.1]{ArtinII} and \stacks[Theorem]{0GIB}}] \label{thm:Artin-contraction} 
    Suppose that $S$ is a $G$-space. Then, every formal contraction problem over $S$ has a solution.
\end{thm}

\subsubsection*{Essential surjectivity in the $G$-scheme case} 

Assume now that $S$ is a $G$-space and let $(U,\mf{X},j)$ be a separated gluing triple over $(S,S_0)$. By Proposition \ref{prop:trip-descent}
 we may assume that $S$ is a separated scheme.

\medskip

\paragraph{Step 1: finite type case} Suppose first that $U\to S^\circ$ and $\mf{X}\to\wh{S}$ are finite type. By Nagata compactification for algebraic spaces (see \cite[Theorem 1.2]{CLO}), we may find a proper morphism of algebraic spaces $\ov{X}\to S$ and a dense open embedding $U\to \ov{X}^\circ$. As $\ov{X}\to S$ is proper, we know from Proposition \ref{prop:j-map} that $j_{\ov{X}}$ is an isomorphism. In particular, we have an open embedding of rigid spaces obtained as the composition
\begin{equation}\label{eq:compactification-open-emb}
    \mf{X}^\mr{rig}\xrightarrow{j}U^\mr{an}\to (\ov{X}^\circ)^\mr{an}\xrightarrow{j_{\ov{X}}^{-1}}\wh{\ov{X}}^\mr{rig}.
\end{equation}
Thus, there exist admissible modifications $\mf{g}\colon \mf{Y}\to \mf{X}$ and $\mf{Y}'\to \wh{\ov{X}}$ and an open embedding $\mf{Y}\to \mf{Y}'$ which recovers \eqref{eq:compactification-open-emb} on applying $(-)^\mr{rig}$. 

By \stacks[Theorem]{0GDU} there exists a unique $U$-modification $Y'\to \ov{X}$  whose completion recovers $\mf{Y}'\to \wh{\ov{X}}$. Let $|Y|\subseteq |Y'|$ be the union of 
\begin{equation*}
    |U|\subseteq |\ov{X}^\circ|=|(Y')^\circ|,\quad\text{and}\quad |\mf{Y}|\subseteq |\mf{Y}'|=|\wh{Y}'|\subseteq |Y'|.
\end{equation*}
Then using Proposition \ref{prop:triple-scheme-same-space-different-day} one sees that $|Y|$ is an open subset of $|Y'|$ and so corresponds to an open algebraic subspace $Y$ of $Y'$ (see \stacks[Lemma]{03BZ}). 

The quadruple $(Y,|\mf{Y}|,\mf{X}, g)$ is then a formal contraction problem in the sense of Definition \ref{def:formal-contraction}. As $S$ is a $G$-space we see from Theorem \ref{thm:Artin-contraction} that this formal contraction problem has a solution in the sense of Definition \ref{thm:Artin-contraction}. Write this solution as $(X,T,g,\iota)$. As $T$ is the image of $|\mf{Y}|$ it is clear that $T=|\mf{X}|$ and so $\mct(X)=(U,\mf{X},j)$ as desired.

\medskip

\paragraph{Step 2: general case} Take coherent open covers $\{U_i \}_{i\in\mc{I}}$ and $\{\mf{X}_j\}_{j\in\mc{J}}$ of $U$ and $\mf{X}$, respectively. For each $j$ in $\mc{J}$ one has that $\mf{X}_{j}^\mr{rig}$ is an open subset of $U^\mr{an}=\bigcup_i U_i^\mr{an}$, and so there exists some finite subset $S_{i(j)}\subseteq \mc{I}$ such that that $\mf{X}_{j}^\mr{rig}\subseteq W_{i(j)}^\mr{an}$ where $W_{i(j)}= \bigcup_{i\in S_{i(j)}}U_{i(j)}$. Then, for each $j$ in $\mc{J}$ and each pair $(j,j')$ in $\mc{J}$ the triples $(W_{i(j)}^\circ,\mf{X}_j,j)$ and $(W_{i(j)}\cap W_{i(j')},\mf{X}_j\cap \mf{X}_{j'},j)$ is a separated open gluing datum with finite type constituents. 

By \textbf{Step 1}, there exist algebraic spaces $X_j$ and $X_{j,j'}$ with
\begin{equation*} 
\mct(X_j)=(W_{i(j)},\mf{X}_j,j),\qquad \mct(X_{j,j'})=(W_{i(j)}\cap W_{i(j')},\mf{X}_j\cap \mf{X}_{j'},j),
\end{equation*}
respectively. Moreover, combining the fully faithfulness argument from \textbf{Step 2} and Proposition~\ref{prop:triples-morphisms-properties}, we obtain a separated open gluing datum of algebraic spaces. Taking the gluing of this datum results in an algebraic space $X$ separated and locally of finite type over $S$ which, by construction, satisfies $\mct(X)=(U,\mf{X},j)$.

\subsection{Proof of Proposition~\ref{prop:triples-morphisms-properties}}
\label{ss:proof-of-morphism-properties}

We recall the statement below for convenience.

\begin{prop*} 
    Let $f\colon Y\to X$ be a morphism of algebraic spaces separated and locally of finite type over $S$, and let $P$ be one of the following properties:
    \begin{multicols}{3}
    \begin{enumerate}[i)]
        \item quasi-compact,
        \item surjective,
        \item open embedding,
        \item closed embedding,
        \item (locally) quasi-finite,
        \item isomorphism,
        \item finite,
        \item separated,
        \item flat,
        \item \'etale,
        \item smooth.
    \end{enumerate}
    \end{multicols}
    \noindent Then, the map $f$ has property $P$ if and only if $\mct(f)$ does.
\end{prop*}

\begin{proof} 
As all these properties are stable/local for the \'etale topology, by considering $X'\times_X Y\to X'$ for an \'etale cover $X'\to X$ where $X'$ is a separated scheme locally of finite type over $S$, we may assume $X$ and $Y$ are representable.

For all of these properties $P$, it is evident if $f$ satisfies $P$ then $f^\circ\colon Y^\circ\to X^\circ$ and $f_n\colon Y_n\to X_n$ satisfies $P$ for all $n$. Thus, to verify that $\mct(f)$ satisfies $P$ it suffices to show that $\wh{f}\colon \wh{Y}\to \wh{X}$ satisfies $P$ if and only if $f_n$ satisfies $P$ for all $n$. But, in all cases this is true as is seen from
\begin{multicols}{3}
\begin{enumerate}[i)]
    \item obvious,
    \item obvious,
    \item Prop.\@ 4.4.2 + xi),
    \item Prop.\@ 4.3.6,
    \item obvious,
    \item obvious,
    \item Prop.\@ 4.2.3,
    \item Prop.\@ 4.6.9,
    \item Prop.\@ 4.8.1,
    \item Prop.\@ 5.3.11,
    \item Prop.\@ 5.3.18,
\end{enumerate}
\end{multicols}
\noindent where each reference is to a result in \cite[Chapter I]{FujiwaraKato}.

So, suppose now that $\mct(f)$ is $P$. The fact that $f$ is $P$ is obvious for i), ii), v). Let us then observe that every other claim follows easily once we establish vii) and ix). Namely, xi) and xii) easily follow once ix) is established from \stacks[Lemma]{01V9}. Moreover, once xi) is known, the case of iii) follows easily: if $\mct(f)$ is an open embedding, then $f$ is \'etale, and so by \stacks[Lemma]{02LC} it suffices to show that $f$ is universally injective. But, by \stacks[Lemma]{01S4} it suffices to show that $\Delta_f$ is surjective, which can clearly be checked on the level of $\mct(f)$. Then, from iii) and ii) we see that vi) follows as an isomorphism is a surjective open embedding. Now suppose that vii) is established. Then, to show iv) it suffices by \stacks[Lemma]{03BB} that $f$ is a monomorphism. But, this means showing that $\Delta_f$ is an isomorphism, which follows easily from vi). Finally, to show viii) we must show that $\Delta_f$ is a closed embedding, but this follows from iv).

To address case ix), we assume $Y=\Spec(B)$ and $X=\Spec(A)$ and show that $A\to B$ is flat. By assumption, $\Spec(B)\mysetminus V(IB)\to\Spec(A)\mysetminus V(IA)$ and $\wh{A}\to\wh{B}$ are flat (see \cite[Chapter I, Proposition 4.8.1]{FujiwaraKato}). Suppose that $\mf{n}$ is a point of $V(IB)$ and let $\mf{m}$ be its image in $V(IA)$. As $A_\mf{m}/I^nA_\mf{m}\to B_\mf{n}/I^nB_\mf{n}$ is flat, it suffices to verify that the conditions of \cite[Chapter 0, Corollary 8.3.9]{FujiwaraKato} are satisfied. Evidently the pair $(B_\mf{n},I)$ is Zariskian, and so it suffices to verify that the pairs $(A_\mf{m},IA_\mf{m})$ and $(B_\mf{n},IB_\mf{n})$ are \textbf{(APf)} as in \cite[Chapter 0, \S7.(c)]{FujiwaraKato}. For type (N) this follows from \cite[Chapter 0, Proposition 7.4.14]{FujiwaraKato}. For Type (V), it suffices by \cite[Chapter 0, Corollary 8.2.17]{FujiwaraKato} to show that $(B_\mf{n},IB_\mf{n})$ is pseudo-adhesive, but by \cite[Chapter 0, Proposition 8.5.7]{FujiwaraKato} it further suffices to show that $(B,IB)$ is pseudo-adhesive. But, $(\mcO,(\varpi))$ is universally pseudo-adhesive (see \cite[Theorem 2.16]{ZavyalovSheafiness} and \cite[Chapter 0, Theorem 9.2.1]{FujiwaraKato}) from where the claim follows.

Finally, to show case vii), i.e., finite, let us observe that we may write $Y^\circ=\underline{\Spec}(\mc{A}^\circ)$ for a~coherent $\mcO_{X^\circ}$-algebra $\mc{A}^\circ$, and similarly we may write $\wh{Y}=\underline{\Spf}(\wh{\mc{A}})$ for a coherent $\mcO_{\wh{X}}$-algebra $\wh{\mc{A}}$ (see \cite[Chapter I, Proposition 4.2.6]{FujiwaraKato}). By setup the pullback of $\mc{A}^\circ$ and $\wh{A}$ to $\wh{X}^\mr{rig}$ is isomorphic, and so by Proposition \ref{prop:coh-gluing} we may find a coherent $\mcO_X$-algebra $\mc{A}$ inducing $\mc{A}^\circ$ and $\wh{A}$. Then, $\underline{\Spec}(\mc{A})\to X$ is a finite morphism, and by construction $\mct(\underline{\Spec}(\mc{A}))$ is isomorphic to $Y$. By Corollary \ref{cor:fully-faithfulness-schemes}, $\underline{\Spec}(\mc{A})$ is isomorphic to $Y$ as an $X$-scheme, from where the claim follows.
\end{proof}

\subsection{Proof of Proposition \ref{prop:affine-gluing-base-case}}\label{ss:proof-of-fg}

We recall the statement below for convenience. 

\begin{prop*} 
     Let $R$ be a Noetherian ring and $(\pi)\subseteq R$ an ideal. We further let
\begin{itemize} 
\item $A$ be a finitely generated $R[\nicefrac{1}{\pi}]$-algebra, 
\item $B$ be a topologically finitely generated $\wh{R}$-algebra, 
\item and $j^*\colon A\to C=B[\nicefrac{1}{\pi}]$ be a map of $R[\nicefrac{1}{\pi}]$-algebras with dense image and for which the induced map $\Spa(C) \to \Spec(A)^{\rm an}$ is an open immersion. 
\end{itemize}
Then the $R$-algebra $D$ defined as the pull-back   
\begin{equation*}
    \begin{tikzcd}
	D & B \\
	A & C,
	\arrow[from=1-1, to=1-2]
	\arrow[from=1-1, to=2-1]
	\arrow["\lrcorner"{anchor=center, pos=0.125}, draw=none, from=1-1, to=2-2]
	\arrow[from=1-2, to=2-2]
	\arrow["j"', from=2-1, to=2-2]
\end{tikzcd}
\end{equation*}
is finitely generated, and satisfies $D[\nicefrac{1}{\pi}]=A$ and $\widehat{D} = B$.
\end{prop*}

\begin{proof} 
We break the proof into four steps.
\medskip

\paragraph{Step 1} We shall find elements $x_1, \ldots, x_n$ in $D$ satisfying the following properties:
\begin{enumerate}
    \item their images in $A$ generate it as an $R\ip$-algebra,
    \item their images in $B$ topologically generate it as an $\wh{R}$-algebra,
    \item the open subset $\Spa(C)$ of $\Spec(A)^{\rm an}$ is cut out by the inequalities $|x_i|\leqslant 1$ ($i=1, \ldots, n$),
    \item the $\pi$-torsion of $B$ (which is the kernel of $B\to C$) is generated as an ideal by the image of a subset $\{x_1, \ldots, x_r\}$, and $x_1, \ldots, x_r$ map to zero in $A$.
\end{enumerate}

Let $\{y'_i\}$ be any finite set of generators for $A$ over $R[\nicefrac{1}{\pi}]$. Note that $\{\pi^N y_i'\}$ is also a set of generators and, as $B$ has open image in $C$, have image in $C$ lying in the image of $B$ for $N\gg 0$. Choose such an $N$ and $\{y_i\}$ be a subset of $D$ mapping to $\{\pi^N y'_i\}$ in $A$. Similarly, let $\{\hat z_j\}$ be a finite set of topological generators for $B$ over $\wh{R}$. Then, for any other elements $\wh{z}'_j$ the set $\{\wh{z}_j+\pi\wh{z}'_j\}$ topologically generates $B$, as $j^*$ has dense image, must also have image in $C$ lying in the image of $A$. Choose such $\{\wh{z}'_j\}$ and let $\{z_j\}$ be a subset of $D$ mapping to $\{\wh{z}_j+\pi\wh{z}'_j\}$. Write $\{u_1,\ldots,u_m\}= \{y_i\}\cup\{z_j\}$.

Observe that by construction the map $R[\nicefrac{1}{\pi}][U_1,\ldots,U_m]\to A$ with $U_\alpha\mapsto u_\alpha$ is surjective. Let us write $I$ for its kernel. Consider the natural map $C'\to C$ where 
\begin{equation}\label{eq:Cprime-to-C} 
    \begin{aligned} C' &= \wh{R}[\nicefrac{1}{\pi}]\langle U_1,\ldots,U_m\rangle/I\wh{R}[\nicefrac{1}{\pi}]\langle U_1,\ldots,U_m\rangle\\ &=A\otimes_{R[\nicefrac{1}{\pi}][U_1,\ldots,U_m]} \wh{R}[\nicefrac{1}{\pi}]\langle U_1,\ldots,U_m\rangle,\end{aligned}
\end{equation}
which is surjective by construction. The space $\Spa(C')$ is a Weierstrass affinoid subdomain of $\Spec(A)^{\rm an}$ defined by $|u_\alpha|_\pi\leqslant 1$. Indeed, $\Spa(C')$ is the pullback of the Weierstrass domain $\Spa(\wh{R}[\nicefrac{1}{\pi}]\langle U_1,\ldots,U_m\rangle)\to \mathbb{A}^{n,\mr{an}}_{\wh{R}[\nicefrac{1}{\pi}]}$ along the Zariski closed embedding $\Spec(A)^\mr{an}\hookrightarrow \mathbb{A}^{n,\mr{an}}_{\wh{R}[\nicefrac{1}{\pi}]}$.

As $\Spa(C)\to \Spec(A)^\mr{an}$ is also an open embedding, we deduce that the map $\Spa(C)\to \Spa(C')$ induced by \eqref{eq:Cprime-to-C} is an open immersion.  But since this morphism is also a closed immersion, we must have a decomposition $C' = C\times C''$ where $C'\to C$ is the first projection. Consider the element $t' = (0, 1/\pi)\in C\times C''=C'$. Since $A\to C'$ has dense image by construction, we may find an element $t_0$ in $A$ whose image is close enough to $t'$ so that $t_0$ is power bounded on $\Spa(C)$ and $t_0$ is not power bounded on $\Spa(C'')$. Observe that, in particular, $j^\ast(t_0)$ lies in the image of $B\to C$, and so we may choose some $t$ in $D$ mapping to it.

Let $\{w_1,\ldots,w_k\}$ be generators of the $\pi$-torsion $B[\pi^\infty]$ which is finitely generated as $R$ is of Type (N). Set $\{x_1,\ldots,x_n\}=\{u_1,\ldots,u_m,t,w_1,\ldots,w_k\}$.

\medskip

\paragraph{Step 2} We shall prove that $\{x_1, \ldots, x_n\}$ in $D$ as in \textbf{Step~1} generate $D$ as an $R$-algebra, assuming that $B$ is $\pi$-torsionfree. Let $x$ be shorthand for $\{x_1,\ldots,x_n\}$ and consider the commutative diagram

\[ 
    \begin{tikzcd}[column sep={1.5cm,between origins},row sep={1.5cm,between origins}]
        & R[x] \arrow[drr] \arrow[ddl] \arrow[rrrr,"\delta"] & & & & D \arrow[drr]\arrow[ddl]\\
        & & & R\langle x\rangle \arrow[rrrr,"\beta",crossing over,two heads] & & & & B\arrow[ddl] \\
        R\ip[x]\arrow[drr]\arrow[rrrr,"\alpha",two heads] & & & & A\arrow[drr] \\
        & & R\ip\langle x\rangle\arrow[rrrr,"\gamma",two heads] \arrow[from=uur,crossing over] & & & & C
    \end{tikzcd}
\]
Our goal is to show that the top arrow $\delta$ is surjective. Fix $d$ in $D$ and denote by $a,b,c$ its images in $A,B,C$, and let $U,V,W$ be their preimages in $R\ip[x]$, $R\langle x\rangle$, $R\ip\langle x\rangle$. Since the left square is cartesian as well, we need to show that the images of $U$ and $V$ in $W$ intersect. To this end, it is enough to show two assertions:
\begin{enumerate}[(i)]
    \item The image of $U$ in $W$ is dense.
    \item The image of $V$ in $W$ is open (and nonempty).
\end{enumerate}
Let us denote by $I$, $J$, $K$ the kernels of $\alpha$, $\beta$, $\gamma$. Thus
\[ 
    U = \tilde\alpha + I, \qquad V = \tilde\beta + J, \qquad W = \tilde\gamma + K
\]
for arbitrary lifts $\tilde\alpha$, $\tilde\beta$, $\tilde\gamma$ of $a,b,c$. In particular, showing (i) and (ii) is equivalent to showing the density of the image of $I$ in $K$ and the openness of the image of $J$ in $K$. 

For the first assertion, we note first that condition (3) means that 
\[ 
    C \simeq A\otimes_{R\ip[x]} R\ip\langle x\rangle.
\]
This is equivalent to saying that $K = I\cdot R\ip\langle x\rangle$. Thus the image of $I$ is dense in $K$. For the second assertion, the assumed injectivity of $B\to C$ implies that $J = K\cap R\langle x\rangle$. Since $R\langle x\rangle$ is open in $R\ip\langle x\rangle$, $J$ is open in $K$.

\medskip

\paragraph{Step 3} 
We now show that $D$ is a finitely generated $R$-algebra in general. Since $B$ is Noetherian, there exists an $N\geqslant 0$ such that $B[\pi^\infty] = B[\pi^N]$. In this case, we have $B \simeq B'\times_{B'''} B''$, where
\begin{equation*}
B'=B/B[\pi^\infty] = {\rm im}(B\to C), \qquad B''=B/\pi^N B, \qquad B''' = B'\otimes_B B'' = B'/\pi^NB'.
\end{equation*}
Therefore, setting $D' = A\times_C B'$, we have
\[ 
    D\isomto D'\times_{B'''} B''.
\]
Now $D'$, $B'$, and $B'''$ are finitely generated $R$-algebras. Moreover, the maps $D\to D'$ and $D\to B''$ are surjective. Therefore by \stacks[Lemma]{00IT}, the ring $D$ is finitely generated over $R$.

\medskip

\paragraph{Step 4} It finally suffices to verify that $A=D[\nicefrac{1}{\pi}]$ and $\wh{D}=B$. By the arguments in \textbf{Step 1}, both $D[\nicefrac{1}{\pi}]\to A$ and $\wh{D}\to B$ are surjective, thus it suffices to verify both maps are injective. Suppose that $d=(a,b)$ belongs to ${\rm ker}(D[\nicefrac{1}{\pi}]\to A)$. Then, it is of the form $d=(0,b)$ and so $b$ is itself $\pi$-torsion. Thus, $d$ is actually zero in $D[\nicefrac{1}{\pi}]$ as desired. To see that $\wh{D}\to\wh{B}$ injective it suffices to show that $D/\pi^n\to B/\pi^n$ is injective for all $n$. But, if $d=(a,b)$ in $D$ maps to zero in $B/\pi^n$ we can write $b=\pi^nb'$ and then $d=\pi^n d'$ where $d'=(\nicefrac{a}{\pi^n},b)$ from where the injectivity of $D/\pi^n\to B/\pi^n$ follows.
\end{proof}

\bibliographystyle{plain} 
\bibliography{bib}

\end{document}